\documentclass[a4paper,11pt,reqno]{amsart}

\usepackage[T1]{fontenc}
\usepackage{amsmath,amssymb,amsthm,mathtools,hyperref,geometry,cleveref,mathrsfs,xcolor,microtype,tabularx,array,multirow,float}

\allowdisplaybreaks

\definecolor{DarkBlue}{rgb}{0.1,0.1,0.55}
\definecolor{DarkRed}{rgb}{0.55,0.1,0.1}

\hypersetup{colorlinks=true,linkcolor=DarkBlue,citecolor=DarkRed,urlcolor=DarkBlue}

\numberwithin{equation}{section}
\theoremstyle{plain}
\newtheorem{theorem}{Theorem}[section]
\newtheorem{proposition}[theorem]{Proposition}
\newtheorem{corollary}[theorem]{Corollary}
\newtheorem{lemma}[theorem]{Lemma}
\newtheorem{remark}[theorem]{Remark}
\newtheorem*{yauthm}{Yau's classification theorem}
\newtheorem*{calabithm}{Calabi's classification theorem}
\newtheorem*{simon}{Simon's conjecture}

\newcommand{\R}{\mathbb{R}}
\newcommand{\tr}{\operatorname{tr}}
\newcommand{\HH}{\mathbf{H}}
\newcommand{\wt}{\tilde}
\newcommand{\mS}{\mathbb{S}}
\newcommand{\Sn}{\mathbb{S}^N}
\newcommand{\B}{\mathcal{\tilde{B}}}
\newcommand{\C}{\mathcal{\tilde{C}}}
\newcommand{\ta}{\tilde{a}}
\newcommand{\tb}{\tilde{b}}
\newcommand{\tri}{\triangle}
\newcommand{\thh}{\tilde h}
\newcommand{\tS}{\tilde S}

\newcommand{\tE}{\tilde E}
\newcommand{\Hess}{\operatorname{Hess}}

\title[Pinching rigidity of PMC surfaces in spheres]{Pinching rigidity of surfaces with parallel mean curvature vector in spheres}

\author[W. R. Ding]{Weiran Ding$^{1}$}
\address{$^{1}$School of Mathematical Sciences, South China Normal University, Guangzhou 510000, P. R. CHINA.}
\email{dingwr0806@m.scnu.edu.cn}

\author[J. Q. Ge]{Jianquan Ge$^{2}$}
\address{$^{2}$School of Mathematical Sciences, Laboratory of Mathematics and Complex Systems, Beijing Normal University, Beijing 100875, P. R. CHINA.}
\email{jqge@bnu.edu.cn}

\author[F. G. Li]{Fagui Li$^{3,*}$}
\address{$^{3,*}$Frontier Interdisciplinary Domain, Beijing Institute of Technology, Zhuhai 519088, P. R. CHINA.}
\email{lifagui@bitzh.edu.cn}

\thanks{* the corresponding author.}

\begin{document}

\begin{abstract}
	Inspired by the Simon conjecture for minimal surfaces in spheres, we study closed surfaces with parallel mean curvature vector and positive Gaussian curvature immersed in unit spheres. Let $h$ be the second fundamental form, let $\HH$ be the mean curvature vector field, and set $\wt h=h-\HH g$ and $\wt S=|\wt h|^2=|h|^2-2H^2$, where $H=|\HH|$ and $g$ is the induced Riemannian metric on the surface $M$. We establish three Simons-type integral identities for $\wt S$, which extend the first, second and third gap identities in the minimal case. As applications, we obtain the first two sharp endpoint gaps and several rigidity and oscillation estimates in the third interval. We further characterize the endpoint cases by combining these identities with the classification theorems of Calabi and Yau.
\end{abstract}

\maketitle

\section{Introduction}

In 1967, E. Calabi \cite{Cala67} studied minimal immersions of $S^2(K)$ with constant Gaussian curvature $K$ into spheres $\Sn\left(\frac{1}{r^2}\right)$ with radius $r$. In particular, we write $\Sn=\Sn(1)$. Such immersions were classified, up to rigid motions, by the discrete values of the curvature. More precisely, Calabi proved the following classification theorem:

\begin{calabithm}[\cite{Cala67}, Theorem 5.2]
	Let $M$ be a $2$-sphere with a Riemannian metric with constant curvature $K$, and let $f:S^2(K)\to\Sn\left(\frac{1}{r^2}\right)\subset\mathbb{R}^{N+1}$ be an isometric, minimal immersion of $M$ into the sphere of radius $r$, such that the image is not contained in any hyperplane of $\mathbb{R}^{N+1}$. Then
	\begin{enumerate}
		\item there exists an integer $s\ge 1$ such that $N=2s$ and the value of $K$ is uniquely determined at the value $$K=\frac{2}{s(s+1)r^2};$$
		\item the immersion $f$ is uniquely determined up to a rigid rotation of $\mathbb{S}^N$, and the $N+1$ components of the vector $f$ are a suitably normalized basis for the spherical harmonics of order $s$ on $M$.
	\end{enumerate}
\end{calabithm}
These examples are usually called Calabi's $2$-spheres. Later, M. do Carmo and N. Wallach \cite{doCa71} and independently S. S. Chern \cite{Cher70} obtained the same classification, using representation theory and the method of moving frames, respectively. Motivated by these results, U. Simon proposed a quantization conjecture \cite{Kozl84,LiSi03} for the Gaussian curvature of closed minimal surfaces in unit spheres.
\begin{simon}
	Let $M$ be a closed surface minimally immersed in  $\mathbb{S}^N$ such that the image is not contained in any hyperplane of $\mathbb{R}^{N+1}$. Using $$2K=2-S,$$ where $S=|h|^2$ is the squared norm of the second fundamental form, if $S(s)\leq S\leq S(s+1)$ for an $s\in\mathbb{N}$, where $$S(s)\coloneqq\frac{2(s-1)(s+2)}{s(s+1)},$$ then either $S=S(s)$ or $S=S(s+1)$, and thus the immersion is one of Calabi's $2$-spheres with the dimension of the ambient space $N=2s$ or $N=2s+2$, respectively.
\end{simon}
The first and second gaps were proved in the classical works of Simon and others \cite{Benk79,Kozl84}, and in our previous paper \cite{Ding25,Ding26} we also gave a unified treatment of the first and second gaps and obtained partial answers to the third gap. We also refer to \cite{Bolt88} for some related results on higher gaps of the Simon conjecture.\par
We also recall another classical rigidity problem closely related to gap phenomena in spheres, namely Chern's conjecture; see, among others, \cite{Chan93a,CDK70,Ding11,Peng83a,Peng83b,Tang23,XuXu17,Yang98}. It predicts that the possible constant values of the scalar curvature of closed minimal hypersurfaces in a unit sphere form a discrete set. In terms of the squared norm of the second fundamental form, this can be interpreted as a discreteness problem for constant values of $S=|h|^2$. There have also been several extensions and related results for hypersurfaces with constant mean curvature in spheres; see, among others, \cite{Alen94,Chan93b,Chen90,LeiX21,XuXu13}. Since constant mean curvature is equivalent to parallel mean curvature vector in the hypersurface case, these results provide one of the motivations for studying gap and rigidity phenomena for submanifolds with parallel mean curvature vector.\par
The purpose of the present paper is to extend the first three gaps to surfaces with parallel mean curvature vector in spheres. Let $M$ be a closed surface immersed in $\Sn$ with mean curvature vector field $\HH$. We assume that $\nabla^{\perp}\HH=0$, namely, $M$ has parallel mean curvature vector. For brevity, surfaces with parallel mean curvature vector will occasionally be called PMC surfaces. In this setting the natural object is no longer the original second fundamental form $h$, but its traceless part $$\wt h=h-\HH g,$$ where $g$ denotes the induced Riemannian metric on the surface $M$. If we write $\wt S=|\wt h|^2$, then $$\wt S=S-2H^2,$$ where $H=|\HH|$ is constant. Moreover, the Gauss equation becomes $$2K=2-S+4H^2=2(1+H^2)-\wt S\coloneqq 2c-\wt S,$$ where $c\coloneqq 1+H^2$. For PMC surfaces, there is a classical reduction theorem of Yau. More precisely, Yau proved the following classification theorem in space forms.

\begin{yauthm}[\cite{Yau74}, Theorem 4]\label{YAU}
	Let $M^2$ be a surface with parallel mean curvature vector in a space form $\mathcal{N}$. Then either $M^2$ is a minimal surface of an umbilical hypersurface of $\mathcal{N}$ or $M^2$ lies in a three-dimensional umbilical submanifold of $\mathcal{N}$ and has constant mean curvature therein.
\end{yauthm}

We also mention that B. Y. Chen independently obtained important classification results for surfaces with parallel mean curvature vector \cite{Chen73a,Chen73b,Chen10}. In this sense, the theory of PMC surfaces with parallel mean curvature in space forms is closely tied to the theory of minimal surfaces. Yau's classification theorem is a more important motivation for the present work. It suggests that the endpoint phenomena for PMC surfaces should be governed by the same Calabi's $2$-spheres, after passing to the appropriate totally umbilical sphere.\par
Combining the classification theorems of Yau and Calabi, we obtain a natural geometric characterization of PMC surfaces. These models determine the endpoint values of the quantity $\wt S$ which appear in the gap problem. We record this consequence first, since it provides the geometric background for the pinching results below.

\begin{proposition}\label{yauchara}
	Let $M^2$ be a closed surface immersed in $\mathbb S^N$ with parallel mean curvature vector. Assume that $\wt S$ is one of the endpoint values $$0,\quad\frac43c,\quad\frac53c,\quad\frac95c,\quad\cdots,\quad\frac{2(s-1)(s+2)}{s(s+1)}c,\quad\cdots.$$ Then, after passing to the orientable double cover if necessary, $M$ is one of the following:
	\begin{enumerate}
		\item $K=1+H^2$. In this case $M$ is totally umbilical in $\Sn$.
		\item $K=\frac{2(1+H^2)}{s(s+1)}$ for $s=2,3,4,\cdots$. Then $M$ is obtained by Calabi's $2$-spheres $$\Psi_s:S^2\left(\frac{2c}{s(s+1)}\right)\to\mathbb{S}^{2s}(c)\subset\Sn,\quad s=2,3,4,\cdots,$$ where the last inclusion is totally umbilical.
	\end{enumerate}
\end{proposition}

Motivated by the classification, we now investigate the corresponding pinching problem. Define $$S_-\coloneqq\frac{27}{59}+\frac{4\sqrt{151705}}{885}\cos\left[\frac13\arccos\left(-\frac{1650478\sqrt{151705}}{920576281}\right)\right]$$ and $$S_+\coloneqq1+\sqrt[3]{\frac{75357+\sqrt{12014474010}}{30375}}+\sqrt[3]{\frac{75357-\sqrt{12014474010}}{30375}}$$ be the unique zero of $$\Theta_-\coloneqq7965x^3-10935x^2-13509x+15295$$ and $$\Theta_+\coloneqq1125x^3-3375x^2+9790x-13122$$ in $\left(\frac53,\frac95\right)$, respectively. Numerically, $$S_-=1.707551424943195707\cdots,\qquad S_+=1.785237928679242784\cdots.$$ Finally, define $$G(x)\coloneqq\frac{120x(3x-4)(9-5x)}{2225x^2-420x-1701},\qquad x\in\left[\frac53,\frac95\right].$$ Under suitable curvature pinching assumptions, our main results verify the first two gaps in the PMC setting and provide additional rigidity results in the range corresponding to the third gap.

\begin{theorem}\label{main1}
	Let $M$ be a closed surface immersed in $\Sn$ with parallel mean curvature vector. Set $\wt S = |\wt h|^2$ and $c=1+H^2$. Then we have (see Theorems \ref{1stint}, \ref{2ndint}, \ref{s3-left-gap}, \ref{s3-interior} and \ref{s3-right-gap})
	\begin{enumerate}
		\item\label{main1.1} if $0\le\wt S\le\frac{4}{3}c$, then $\wt S\equiv 0$ or $\wt S\equiv\frac{4}{3}c$. Moreover, if $\wt S\equiv 0$, then $M$ is totally umbilical in $\Sn$ with $K\equiv 1+H^2$. If $\wt S\equiv\frac{4}{3}c$, then $M$ is Calabi's $2$-sphere with $K\equiv\frac{1}{3}(1+H^2)$;
		\item\label{main1.2} if $\frac{4}{3}c\le\wt S\le\frac{5}{3}c$, then $\wt S\equiv\frac{4}{3}c$ or $\wt S\equiv\frac{5}{3}c$. Moreover, if $\wt S\equiv\frac{4}{3}c$, then $M$ is Calabi's $2$-sphere with $K\equiv\frac{1}{3}(1+H^2)$. If $\wt S\equiv\frac{5}{3}c$, then $M$ is Calabi's $2$-sphere with $K\equiv\frac{1}{6}(1+H^2)$;
		\item if $\frac{5}{3}c\le\wt S\le\frac{9}{5}c$, then we have the followings:
			\begin{itemize}
				\item[(3a)] if $\frac53c\le \tS<cS_-<\frac95c$, then $\tS\equiv\frac53c$, and $M$ is Calabi's $2$-sphere with $K\equiv\frac{1}{6}(1+H^2)$;
				\item[(3b)] if $\frac53c\le\tS\le\frac95c$ and $\tS\not\equiv\frac53c$, then
					\begin{equation*}
						\tS_{\max}-\tS_{\min}\ge cG\left(\frac{\tS_{\min}}{c}\right)=\frac{120\tS_{\min}(9c-5\tS_{\min})(3\tS_{\min}-4c)}{2225\tS_{\min}^2-420c\tS_{\min}-1701c^2},
					\end{equation*}
					where $\tS_{\max}=\sup_{p\in M}\tS(p)$ and $\tS_{\min}=\inf_{p\in M}\tS(p)$;
				\item[(3c)] if $cS_+<\tS\le \frac95c$, then $\tS\equiv\frac95c$, and $M$ is Calabi's $2$-sphere with $K\equiv\frac{1}{10}(1+H^2)$.
			\end{itemize}
	\end{enumerate}
\end{theorem}

Equivalently, since $S=\wt S+2H^2$, we have the following formulation in terms of the original second fundamental form. The conclusions of Theorem \ref{main1} can be summarized in terms of both $\wt S$ and $S$ as follows.

\begin{table}[H]
	\centering
	\caption{Endpoint values and geometric models.}
	\vspace{0.2cm}
	\renewcommand{\arraystretch}{1.45}
	\begin{tabular}{%
		>{\centering\arraybackslash}m{2cm}|
		>{\centering\arraybackslash}m{6cm}|
		>{\centering\arraybackslash}m{2.5cm}|
		>{\centering\arraybackslash}m{3cm}}
		\hline
		Case & Pinching interval & Endpoint value & Geometric model \\
		\hline
		\multirow{2}{1.1cm}{\centering First gap} & \multirow{2}{6cm}{\centering $0\leq \wt S\leq \frac{4}{3}c$ \\ $2H^2\leq S\le\frac{4}{3}+\frac{10}{3}H^2$} & $\wt S\equiv 0$ $S\equiv 2H^2$ & totally umbilical with $K\equiv c$ \\ \cline{3-4} & & $\wt S\equiv\frac{4}{3}c$ $S\equiv\frac{4}{3}+\frac{10}{3}H^2$ & Calabi's $2$-sphere with $K\equiv\frac{c}{3}$ \\
		\hline
		\multirow{2}{1.1cm}{\centering Second gap} & \multirow{2}{6cm}{\centering $\frac{4}{3}c\le\wt S\le\frac{5}{3}c$ \\ $\frac{4}{3}+\frac{10}{3}H^2\le S\leq \frac{5}{3}+\frac{11}{3}H^2$} & $\wt S\equiv\frac{4}{3}c$ $S\equiv\frac{4}{3}+\frac{10}{3}H^2$ & Calabi's $2$-sphere with $K\equiv\frac{c}{3}$ \\ \cline{3-4} & & $\wt S\equiv\frac{5}{3}c$ $S\equiv\frac{5}{3}+\frac{11}{3}H^2$ & Calabi's $2$-sphere with $K\equiv\frac{c}{6}$ \\
		\hline
		Third gap (left) & $\frac{5}{3}c\le\wt S<cS_-$ \newline $\frac{5}{3}+\frac{11}{3}H^2\le S<(1+H^2)S_-+2H^2$ & $\wt S\equiv\frac{5}{3}c$ $S\equiv\frac{5}{3}+\frac{11}{3}H^2$ & Calabi's $2$-sphere with $K\equiv\frac{c}{6}$ \\
		\hline
		Third gap (right) & $cS_+<\wt S\le\frac{9}{5}c$ \newline $(1+H^2)S_++2H^2<S\le\frac{9}{5}+\frac{19}{5}H^2$ & $\wt S\equiv\frac{9}{5}c$ $S\equiv\frac{9}{5}+\frac{19}{5}H^2$ & Calabi's $2$-sphere with $K\equiv\frac{c}{10}$ \\
		\hline
	\end{tabular}
\end{table}

A claim of Theorem \ref{main1} \eqref{main1.1}-\eqref{main1.2} was mentioned in \cite{XuXu24,XuXupre} by H. W. Xu and Z. Y. Xu, using Yau's classification theorem \cite{Yau74}, while our proof is based on Simons-type integral identities. Yau's classification theorem is then used to identify the geometric models at the endpoints. A related first gap result was obtained by Li and Simon \cite{LiSi03} in a more general setting. More precisely, we obtain the following three formulas, which are the PMC analogues of the first, second, and third gap formulas in our previous paper \cite{Ding25, Ding26}. When $H=0$, they recover the corresponding formulas for minimal surfaces in spheres.

\begin{theorem}\label{main3}
	Let $M$ be a closed surface immersed in $\Sn$ with parallel mean curvature vector and positive Gaussian curvature. Then we have (see Theorems \ref{1stint}, \ref{2ndint}, \ref{3rdint} and \ref{new3rdint})
	\begin{enumerate}
		\item the first integral formula $$\int_M\wt S(3\wt S-4c)=2\int_M\B_1\ge 0;$$
		\item the second integral formula $$\int_M\wt S(3\wt S-4c)(3\wt S-5c)=2\int_M\left[\B_2-\frac{1}{4}\wt S(3\wt S-4c)^2+\frac{1}{2}|\nabla\wt S|^2\right]\ge 0;$$
		\item the third integral formula
			\begin{equation*}
				\begin{aligned}
					&\hspace{1.3em}\int_M\tS(3\tS-4c)(3\tS-5c)(5\tS-9c)\\
					&=2\int_M\left[\B_3-\frac18\tS(3\tS-4c)(45\tS^2-144c\tS+116c^2)+\frac18(65\tS-166c)|\nabla\tS|^2-\frac58(\triangle\tS)^2\right]\\
					&\ge\int_M\left[\frac34(25\tS-42c)|\nabla\tS|^2-\frac54(\tri\tS)^2\right];
				\end{aligned}
			\end{equation*}
	\end{enumerate}
	where $\B_1=|\nabla\wt h|^2$, $\B_2=|\nabla^2\wt h|^2$ and $\B_3=|\nabla^3\wt h|^2$ are the squared lengths of the first, second and third covariant derivatives of the second fundamental form, respectively.
\end{theorem}

Set $$\varepsilon_0\coloneqq G(S_+)=\frac{120S_+(3S_+-4)(9-5S_+)}{2225S_+^2-420S_+-1701}.$$ Numerically, $$\varepsilon_0=0.004619590030505\cdots>\frac1{217}.$$ By Theorem \ref{main1}, we have
\begin{theorem}\label{oscillation}
	Let $M$ be a closed surface immersed in $\Sn$ with parallel mean curvature vector. Assume that $\frac53+\frac{11}{3}H^2\le S\le\frac95+\frac{19}{5}H^2$. If $$S_{\max}-S_{\min}<\varepsilon_0(1+H^2),$$ then $S\equiv\frac53+\frac{11}{3}H^2$ or $S\equiv\frac95+\frac{19}{5}H^2$.
\end{theorem}

\begin{corollary}\label{noimmersion}
	Let $M$ be a closed surface immersed in $\Sn$ with parallel mean curvature vector. Then there exists no isometric immersion $\varPhi:M\to\mathbb{S}^N$ with parallel mean curvature vector, for any $N$, such that $$\frac{5}{3}+\frac{11}{3}H^2\le S\le \frac{9}{5}+\frac{19}{5}H^2,\quad S\not\equiv \frac{5}{3}+\frac{11}{3}H^2,$$ and $$S_{\max}<S_{\min}+\frac{120(S_{\min}-2H^2)(3S_{\min}-4-10H^2)(9-5S_{\min}+19H^2)}{2225(S_{\min}-2H^2)^2-420(1+H^2)(S_{\min}-2H^2)-1701(1+H^2)^2}.$$
\end{corollary}

The rest of the paper is organized as follows. In Section 2 we collect local formulas for the traceless second fundamental form $\wt h$. In Section 3 we establish an improved Simons identity and prove the first gap pinching theorem. In Section 4 we derive the second Simons-type integral identity and prove the second gap pinching theorem. In Section 5 we calculate the third Simons-type integral formula and give the rigidity results in the third gap.

\section{Notation and local formulas}

Let $M$ be a $2$-dimensional manifold immersed in a unit sphere $\Sn$. We assume the range of indices as follows:
\begin{equation*}
	\begin{aligned}
		1\le i,j,k,&\cdots\le 2,\\
		3\le\alpha,\beta,\gamma,&\cdots\le N,\\
		1\le A,B,C,&\cdots\le N.
	\end{aligned}
\end{equation*}
Let $\{e_1,\cdots,e_N\}$ be a local orthonormal frame field on $T(\Sn)$ such that, when restricted to $M$, $\{e_1,e_2\}$ are tangent to $M$ and $\{e_3,\ldots,e_N\}$ are normal to $M$. Denote by $(\omega_A)$ and $(\omega_{AB})$ the dual coframe field and the connection form field associated to this frame. Let $$\omega_{i\alpha} = h_{ij}^{\alpha} \omega_j.$$ Then $h_{ij}^{\alpha} = h_{ji}^{\alpha}$, and the second fundamental form is given by $$h=h_{ij}^{\alpha}\, \omega_i\omega_j e_{\alpha}.$$ The mean curvature vector field is defined by $$\HH = \frac{1}{2}h_{ii}^{\alpha} e_{\alpha},$$ with length $H=|\HH|$. Throughout the paper we assume that $M$ has parallel mean curvature vector, namely, $$\nabla^{\perp}\HH=0.$$ Define the traceless second fundamental form by $$\wt h_{ij}^{\alpha}=h_{ij}^{\alpha}-H^{\alpha}\delta_{ij}.$$ Then we have $$\wt h_{11}^{\alpha}=-\wt h_{22}^{\alpha},\quad\wt h_{12}^{\alpha}=\wt h_{21}^{\alpha},\quad\text{and}\quad\wt h_{ii}^{\alpha}=0.$$ We introduce column vectors $\ta=(\ta^{\alpha})$ and $\tb=(\tb^{\alpha})$ in $\R^{N-2}$ by $$\ta^{\alpha}=\wt h_{11}^{\alpha}=-\wt h_{22}^{\alpha},\quad\tb^{\alpha}=\wt h_{12}^{\alpha}=\wt h_{21}^{\alpha}.$$ Set
\begin{equation*}
	\begin{aligned}
		\wt S_{\alpha}&=
			\begin{pmatrix}
				\ta^{\alpha} & \tb^{\alpha}\\
				\tb^{\alpha} & -\ta^{\alpha}
			\end{pmatrix},\\
		\wt A&=(\langle\wt S_{\alpha},\wt S_{\beta}\rangle)=2\ta\ta^t+2\tb\tb^t,\\
		\wt S&=|\wt h|^2=\tr\wt A=2(|\ta|^2+|\tb|^2),\quad\text{and}\quad\\
		\wt\rho^{\perp}&=\sum_{\alpha,\beta}|[\wt S_{\alpha},\wt S_{\beta}]|^2.
	\end{aligned}
\end{equation*}
Since $\nabla^{\perp}\HH=0$, the covariant derivatives of $\wt h$ agree with those of $h$. We define
\begin{align*}
	\wt h_{ijk}^{\alpha}\omega_k&=d\wt h_{ij}^{\alpha}+\wt h_{mj}^{\alpha}\omega_{mi}+\wt h_{im}^{\alpha}\omega_{mj}+\wt h_{ij}^{\beta}\omega_{\beta\alpha},\\
	\wt h_{ijkl}^{\alpha}\omega_l&=d\wt h_{ijk}^{\alpha}+\wt h_{mjk}^{\alpha}\omega_{mi}+\wt h_{imk}^{\alpha}\omega_{mj}+\wt h_{ijm}^{\alpha}\omega_{mk}+\wt h_{ijk}^{\beta}\omega_{\beta\alpha},\\
	\wt h_{ijklm}^\alpha\omega_m&=d\wt h_{ijkl}^\alpha+\wt h_{njkl}^\alpha\omega_{ni}+\wt h_{inkl}^\alpha\omega_{nj}+\wt h_{ijnl}^\alpha\omega_{nk}+\wt h_{ijkn}^\alpha\omega_{nl}+\wt h_{ijkl}^\beta\omega_{\beta\alpha},\\
	\wt h_{ijklmn}^\alpha\omega_n&=d\wt h_{ijklm}^\alpha+\wt h_{pjklm}^\alpha\omega_{pi}+\wt h_{ipklm}^\alpha\omega_{pj}+\wt h_{ijplm}^\alpha\omega_{pk}+\wt h_{ijkpm}^\alpha\omega_{pl}+\wt h_{ijklp}^\alpha\omega_{pm}+\wt h_{ijklm}^\beta\omega_{\beta\alpha}.
\end{align*}
For convenience, write $$\ta_i=(\ta_i^{\alpha})\coloneqq(\wt h_{11i}^{\alpha}),\quad\tb_i=(\tb_i^{\alpha})\coloneqq(\wt h_{12i}^{\alpha}),\quad\ta_{ij}=(\ta_{ij}^\alpha)\coloneqq(\wt h_{11ij}^\alpha),$$ and define
\begin{equation*}
	\begin{aligned}
		\B_1&=\sum_{i,j,k,\alpha}(\wt h_{ijk}^{\alpha})^2=|\nabla\wt h|^2,\\
		\B_2&=\sum_{i,j,k,l,\alpha}(\wt h_{ijkl}^{\alpha})^2=|\nabla^2\wt h|^2,\\
		\B_3&=\sum_{i,j,k,l,m,\alpha}(\wt h_{ijklm}^{\alpha})^2=|\nabla^3\wt h|^2.
	\end{aligned}
\end{equation*}
Since $M$ is $2$-dimensional, its Riemann curvature tensor is
\begin{equation}\label{cur1}
	R_{ijkl}=\frac{1}{2}(2c-\wt S)(\delta_{ik}\delta_{jl}-\delta_{il}\delta_{jk}).
\end{equation}
For the normal curvature, because $h_{ij}^{\alpha}=\wt h_{ij}^{\alpha}+H^{\alpha}\delta_{ij}$ and scalar matrices commute, we have
\begin{equation*}
	R_{\alpha\beta kl}=\wt h_{km}^{\alpha}\wt h_{ml}^{\beta}-\wt h_{km}^{\beta}\wt h_{ml}^{\alpha}.
\end{equation*}
In particular,
\begin{equation}\label{cur2}
	R_{\alpha\beta 12}=-R_{\alpha\beta 21}=2(\ta^{\alpha}\tb^{\beta}-\ta^{\beta}\tb^{\alpha}),
\end{equation}
and its first covariant derivative is
\begin{equation}\label{cur3}
	R_{\alpha\beta 12k}=2(\tb^{\beta}\ta_k^{\alpha}+\ta^{\alpha}\tb_k^{\beta}-\tb^{\alpha}\ta_k^{\beta}-\ta^{\beta}\tb_k^{\alpha}).
\end{equation}
The Codazzi equation and the Ricci's formulas are
\begin{align}
	\wt h^\alpha_{ijk}-\wt h^\alpha_{ikj}&=0,\label{cod}\\
	\wt h^\alpha_{ijkl}-\wt h^\alpha_{ijlk}&=\wt h^\alpha_{pj}R_{pikl}+\wt h^\alpha_{ip}R_{pjkl}+\wt h^\beta_{ij}R_{\beta\alpha kl},\label{ric1}\\
	\wt h_{ijklm}^\alpha-\wt h_{ijkml}^\alpha&=\wt h_{pjk}^\alpha R_{pilm}+\wt h_{ipk}^\alpha R_{pjlm}+\wt h_{ijp}^\alpha R_{pklm}+\wt h_{ijk}^\beta R_{\beta\alpha lm},\label{ric2}\\
	\wt h_{ijklmn}^\alpha-\wt h_{ijklnm}^\alpha&=\wt h_{pjkl}^\alpha R_{pimn}+\wt h_{ipkl}^\alpha R_{pjmn}+\wt h_{ijpl}^\alpha R_{pkmn}+\wt h_{ijkp}^\alpha R_{plmn}+\wt h_{ijkl}^\beta R_{\beta\alpha mn}.\label{ric3}
\end{align}
Hence the Codazzi equation \eqref{cod} gives that 
\begin{equation*}
	\ta_2=\tb_1,\,\tb_2=-\ta_1,\text{ and }\B_1=4(|\ta_1|^2+|\ta_2|^2).
\end{equation*}
The Gauss equation yields that
\begin{equation*}
	2K=2-|h|^2+4H^2=2(1+H^2)-\wt S=2c-\wt S.
\end{equation*}
The Laplacian of $\wt h$ is defined by $$\triangle\wt h_{ij}^{\alpha}=\wt h_{ijkk}^{\alpha}.$$ By Ricci's formulas and $\wt h_{mm}^{\alpha}=0$, we obtain that
\begin{equation*}
	\triangle\wt h_{ij}^{\alpha}=(2c-\wt S)\wt h_{ij}^{\alpha}+\wt h_{mi}^{\beta}R_{\beta\alpha jm}.
\end{equation*}
Consequently, we have
\begin{equation}
	\begin{aligned}
		\triangle\ta^{\alpha}&=(2c-\wt S)\ta^{\alpha}+\tb^{\beta}R_{\beta\alpha 12}=(2c-\wt S-2|\tb|^2)\ta^{\alpha}+2\langle\ta,\tb\rangle\tb^{\alpha},\label{lapa}
	\end{aligned}
\end{equation}
and
\begin{equation}
	\begin{aligned}
		\triangle\tb^{\alpha}&=(2c-\wt S)\tb^{\alpha}+\ta^{\beta}R_{\beta\alpha 21}=(2c-\wt S-2|\ta|^2)\tb^{\alpha}+2\langle\ta,\tb\rangle\ta^{\alpha}.\label{lapb}
	\end{aligned}
\end{equation}

\section{The first gap}
We first establish the following lemmas.
\begin{lemma}[Simons' identity \cite{Simo68}]\label{simons}
	Suppose that $M$ is a closed surface immersed in $\Sn$ with parallel mean curvature vector. Under the foregoing assumptions and notation, we have
	\begin{equation*}
		\frac{1}{2}\triangle\wt S=\B_1+2c\wt S-|\wt A|^2-\wt\rho^{\perp}.
	\end{equation*}
\end{lemma}
\begin{proof}
	A direct calculation gives that
	\begin{equation*}
		\begin{aligned}
			\frac{1}{2}\triangle\wt S
			&=\frac{1}{2}\triangle\left(2|\ta|^2+2|\tb|^2\right)\\
			&=\B_1+2\langle\ta,\triangle\ta\rangle+2\langle\tb,\triangle\tb\rangle,
		\end{aligned}
	\end{equation*}
	Combining with \eqref{lapa} and \eqref{lapb}, we prove the lemma.
\end{proof}

\begin{lemma}\label{ineq}
	Suppose that $M$ is a closed surface immersed in $\Sn$ with parallel mean curvature vector. Under the foregoing assumptions and notation, we have
	\begin{equation*}
		-\wt S^2+|\wt A|^2+\wt\rho^{\perp}\le\frac{1}{2}\wt S^2,
	\end{equation*}
	where the equality holds if and only if $\ta\perp\tb$ and $|\ta|=|\tb|$.
\end{lemma}
\begin{proof}
	A direct computation gives that
	\begin{equation*}
		|\wt A|^2=4|\ta|^4+4|\tb|^4+8\langle\ta,\tb\rangle^2,
	\end{equation*}
	and
	\begin{equation*}
		\wt\rho^{\perp}=16\left(|\ta|^2|\tb|^2-\langle\ta,\tb\rangle^2\right).
	\end{equation*}
	Hence
	\begin{equation*}
		-\wt S^2+|\wt A|^2+\wt\rho^{\perp}=8\left(|\ta|^2|\tb|^2-\langle\ta,\tb\rangle^2\right),
	\end{equation*}
	which yields that
	\begin{equation*}
		\frac{1}{2}\wt S^2-\left(-\wt S^2+|\wt A|^2+\wt\rho^{\perp}\right)=2\left(|\ta|^2-|\tb|^2\right)^2+8\langle\ta,\tb\rangle^2\ge 0.
	\end{equation*}
	Therefore the desired inequality follows, and equality holds when $\ta\perp\tb$ and $|\ta|=|\tb|$.
\end{proof}

We now establish an improved version of the Simons identity for parallel mean curvature surfaces. This is the key step in the proof of the first and second gaps.

\begin{theorem}\label{1stlap}
	Suppose that $M$ is a closed surface immersed in $\Sn$ with parallel mean curvature vector and positive Gaussian curvature. Then
	\begin{equation*}
		\frac{1}{2}\triangle\wt S=\B_1-\frac{1}{2}\wt S(3\wt S-4c).
	\end{equation*}
\end{theorem}
\begin{proof}
	Let $$D=-\wt S^2+|\wt A|^2+\rho^{\perp}.$$ By Ricci's formula \eqref{ric1} we have
	\begin{equation}\label{thm1eq}
		\begin{aligned}
			(\wt h_{ijk}^{\alpha}\tri\wt h_{ij}^{\alpha})_k&=\sum_{i,j,\alpha}(\tri\wt h_{ij}^{\alpha})^2+\wt h_{ijk}^{\alpha}\wt h_{pik}^{\alpha}R_{pljl}+\wt h_{ijk}^{\alpha}\wt h_{lpk}^{\alpha}R_{pijl} \\
			&\qquad+\wt h_{ijk}^{\alpha}\wt h_{pi}^{\alpha}R_{pljlk}+\wt h_{ijk}^{\alpha}\wt h_{lp}^{\alpha}R_{pijlk}+\wt h_{ijk}^{\alpha}\wt h_{lik}^{\beta}R_{\beta\alpha jl}+\wt h_{ijk}^{\alpha}\wt h_{li}^{\beta}R_{\beta\alpha jlk}.
		\end{aligned}
	\end{equation}
	To calculate the first term of \eqref{thm1eq}, we set $x\coloneqq|\ta|^2,y\coloneqq|\tb|^2,q\coloneqq\langle\ta,\tb\rangle,m\coloneqq2c-\wt S.$ Then by \eqref{lapa} and \eqref{lapb} we have $$\tri\ta=(m-2y)\ta+2q\tb,\text{ and }\tri\tb=(m-2x)\tb+2q\ta.$$ Hence
	\begin{equation*}
		\begin{aligned}
			|\tri\ta|^2&=\langle(m-2y)\ta+2q\tb,(m-2y)\ta+2q\tb\rangle\\
			&=(m-2y)^2|\ta|^2+4q^2|\tb|^2+4q(m-2y)\langle\ta,\tb\rangle\\
			&=x(m-2y)^2+4q^2y+4q^2(m-2y)\\
			&=x(m-2y)^2+4(m-y)q^2,
		\end{aligned}
	\end{equation*}
	and similarly,
	\begin{equation*}
		\begin{aligned}
			|\tri\tb|^2&=\langle(m-2x)\tb+2q\ta,(m-2x)\tb+2q\ta\rangle\\
			&=(m-2x)^2|\tb|^2+4q^2|\ta|^2+4q(m-2x)\langle\ta,\tb\rangle\\
			&=y(m-2x)^2+4q^2x+4q^2(m-2x) \\
			&=y(m-2x)^2+4(m-x)q^2.
		\end{aligned}
	\end{equation*}
	Therefore we obtain
	\begin{equation*}
		\begin{aligned}
			2\left(|\tri\ta|^2+|\tri\tb|^2\right)&=2(x+y)m^2-16mxy+8xy(x+y)+8(2m-x-y)q^2\\
			&=\wt S(2c-\wt S)^2+8\left[-2mxy+(x+y)xy+2mq^2-(x+y)q^2\right]\\
			&=\wt S(2c-\wt S)^2+8\left(2m-(x+y)\right)(q^2-xy)\\
			&=\wt S(2c-\wt S)^2+8\left(5(x+y)-4c\right)(xy-q^2).
		\end{aligned}
	\end{equation*}
	Now, using $\wt S=2(x+y)$ and $D=8(xy-q^2)$, we get $10(x+y)-8c=5\wt S-8c$, and hence
	\begin{equation*}
		\begin{aligned}
			\sum_{i,j,\alpha}(\tri\wt h^\alpha_{ij})^2&=2\sum_\alpha\left((\tri a^\alpha)^2+(\tri b^\alpha)^2\right)\\
			&=\wt S(2c-\wt S)^2+\frac{1}{2}(5\wt S-8c)D.
		\end{aligned}
	\end{equation*}
	To calculate the second, third and sixth terms of \eqref{thm1eq}, set $$I_2\coloneqq\wt h_{ijk}^{\alpha}\wt h_{pik}^{\alpha}R_{pljl},\;I_3\coloneqq\wt h_{ijk}^{\alpha}\wt h_{lpk}^{\alpha}R_{pijl},\text{ and }I_6\coloneqq\wt h_{ijk}^{\alpha}\wt h_{lik}^{\beta}R_{\beta\alpha jl}.$$ By \eqref{cur1} we have $$R_{pljl}=\frac{1}{2}(2c-\wt S)(\delta_{pj}\delta_{ll}-\delta_{pl}\delta_{jl})=\frac{1}{2}(2c-\wt S)\delta_{pj},$$ hence $$I_2=\frac{1}{2}(2c-\wt S)\wt h_{ijk}^{\alpha}\wt h_{jik}^{\alpha}=\frac{1}{2}(2c-\wt S)\B_1.$$ Similarly,
	\begin{equation*}
		\begin{aligned}
			I_3&=\frac{1}{2}(2c-\wt S)\wt h_{ijk}^{\alpha}\wt h_{lpk}^{\alpha}(\delta_{pj}\delta_{il}-\delta_{pl}\delta_{ij}) \\
			&=\frac{1}{2}(2c-\wt S)\left(\wt h_{ijk}^{\alpha}\wt h_{ijk}^{\alpha}-\wt h_{ijk}^{\alpha}\wt h_{ppk}^{\alpha}\delta_{ij}\right)\\
			&=\frac{1}{2}(2c-\wt S)\B_1.
		\end{aligned}
	\end{equation*}
	By \eqref{cur2} and \eqref{cod} we get
	\begin{equation*}
		\begin{aligned}
			I_6&=\wt h_{i1k}^{\alpha}\wt h_{i2k}^{\beta}R_{\beta\alpha12}+\wt h_{i2k}^{\alpha}\wt h_{i1k}^{\beta}R_{\beta\alpha21} \\
			&=2\wt h_{i1k}^{\alpha}\wt h_{i2k}^{\beta}R_{\beta\alpha12}\\
			&=2(\wt h_{111}^{\alpha}\wt h_{121}^{\beta}+\wt h_{112}^{\alpha}\wt h_{122}^{\beta}+\wt h_{211}^{\alpha}\wt h_{221}^{\beta}+\wt h_{212}^{\alpha}\wt h_{222}^{\beta})R_{\beta\alpha12}\\
			&=8(\ta_1^\alpha\ta_2^\beta-\ta_2^\alpha\ta_1^\beta)(\ta^\beta\tb^\alpha-\ta^\alpha\tb^\beta) \\
			&=16\langle\ta,\ta_2\rangle\langle\tb,\ta_1\rangle-16\langle\ta,\ta_1\rangle\langle\tb,\ta_2\rangle
		\end{aligned}
	\end{equation*}
	Hence we conclude that
	\begin{equation*}
		I_2+I_3+I_6=(2c-\wt S)\B_1+16\langle\ta,\ta_2\rangle\langle\tb,\ta_1\rangle-16\langle\ta,\ta_1\rangle\langle\tb,\ta_2\rangle.
	\end{equation*}
	To calculate the fourth and fifth terms of \eqref{thm1eq}, set $$I_4\coloneqq\wt h_{ijk}^{\alpha}\wt h_{pi}^{\alpha}R_{pljlk}\text{ and }I_5\coloneqq\wt h_{ijk}^{\alpha}\wt h_{lp}^{\alpha}R_{pijlk}.$$ By \eqref{cur1} we have $$I_4=\wt h_{ijk}^{\alpha}\wt h_{pi}^{\alpha}R_{pljlk}=-\frac{1}{2}\wt h_{ijk}^{\alpha}\wt h_{pi}^{\alpha}\delta_{pj}\wt S_k=-\frac{1}{2}\wt h_{ijk}^{\alpha}\wt h_{ij}^{\alpha}\wt S_k.$$ Similarly,
	\begin{equation*}
		\begin{aligned}
			I_5&=\wt h_{ijk}^{\alpha}\wt h_{lp}^{\alpha}R_{pijlk}\\
			&=-\frac{1}{2}\wt h_{ijk}^{\alpha}\wt h_{lp}^{\alpha}(\delta_{pj}\delta_{il}-\delta_{pl}\delta_{ij})\wt S_k\\
			&=-\frac{1}{2}\wt h_{ijk}^{\alpha}\left(\wt h_{ij}^{\alpha}-\wt h_{pp}^{\alpha}\delta_{ij}\right)\wt S_k\\
			&=-\frac{1}{2}\wt h_{ijk}^{\alpha}\wt h_{ij}^{\alpha}\wt S_k.
		\end{aligned}
	\end{equation*}
	Hence we obtain $$I_4+I_5=-\wt h_{ijk}^{\alpha}\wt h_{ij}^{\alpha}\wt S_k=-\frac{1}{2}|\nabla\wt S|^2.$$ For the last term of \eqref{thm1eq}, set $$I_7\coloneqq\wt h_{ijk}^{\alpha}\wt h_{li}^{\beta}R_{\beta\alpha jlk}.$$ Then by \eqref{cur3} and \eqref{cod} we have
	\begin{equation*}
		\begin{aligned}
			I_7&=\wt h_{ijk}^{\alpha}\wt h_{li}^{\beta}R_{\beta\alpha jlk}\\
			&=\wt h_{i1k}^{\alpha}\wt h_{2i}^{\beta}R_{\beta\alpha12k}+\wt h_{i2k}^{\alpha}\wt h_{1i}^{\beta}R_{\beta\alpha21k} \\
			&=\left(\wt h_{i1k}^{\alpha}\wt h_{2i}^{\beta}-\wt h_{i2k}^{\alpha}\wt h_{1i}^{\beta}\right)R_{\beta\alpha12k}\\
			&=2(b^\beta a_1^\alpha-a^\beta a_2^\alpha)R_{\beta\alpha121}+2(a^\beta a_1^\alpha+b^\beta a_2^\alpha)R_{\beta\alpha122}\\
			&=4\left(\langle\tb,\ta_1\rangle^2+\langle\ta,\ta_2\rangle^2-2\langle\ta,\ta_1\rangle\langle\tb,\ta_2\rangle-|\ta|^2|\ta_2|^2-|\tb|^2|\ta_1|^2\right)\\
			&\qquad+4\left(\langle\ta,\ta_1\rangle^2+\langle\tb,\ta_2\rangle^2+2\langle\ta,\ta_2\rangle\langle\tb,\ta_1\rangle-|\ta|^2|\ta_1|^2-|\tb|^2|\ta_2|^2\right)\\
			&=4\left(\langle\ta,\ta_1\rangle^2+\langle\tb,\ta_1\rangle^2+\langle\ta,\ta_2\rangle^2+\langle\tb,\ta_2\rangle^2\right) \\
			&\qquad+8\langle\ta,\ta_2\rangle\langle\tb,\ta_1\rangle-8\langle\ta,\ta_1\rangle\langle\tb,\ta_2\rangle -4(|\ta|^2+|\tb|^2)(|\ta_1|^2+|\ta_2|^2).
		\end{aligned}
	\end{equation*}
	By \eqref{cod}, we have
	\begin{equation*}
		\wt S_1=4\left(\langle\ta,\ta_1\rangle+\langle\tb,\ta_2\rangle\right)\text{ and }\wt S_2=4\left(\langle\ta,\ta_2\rangle-\langle\tb,\ta_1\rangle\right).
	\end{equation*}
	Hence $$-4(|\ta|^2+|\tb|^2)(|\ta_1|^2+|\ta_2|^2)=-\frac{1}{2}\wt S\B_1,$$ and $$\frac14|\nabla\wt S|^2=4\left(\langle\ta,\ta_1\rangle^2+\langle\tb,\ta_1\rangle^2+\langle\ta,\ta_2\rangle^2+\langle\tb,\ta_2\rangle^2\right)+8\langle\ta,\ta_1\rangle\langle\tb,\ta_2\rangle-8\langle\ta,\ta_2\rangle\langle\tb,\ta_1\rangle.$$ Then we have
	\begin{equation*}
		I_7=-\frac{1}{2}\wt S\B_1+\frac{1}{4}|\nabla\wt S|^2+16\langle\ta,\ta_2\rangle\langle\tb,\ta_1\rangle-16\langle\ta,\ta_1\rangle\langle\tb,\ta_2\rangle.
	\end{equation*}
	Therefore,
	\begin{equation*}
		\begin{aligned}
			(\wt h_{ijk}^{\alpha}\tri\wt h_{ij}^{\alpha})_k&=\frac{1}{2}(4c-3\wt S)\B_1+\wt S(2c-\wt S)^2+\frac{1}{2}(5\wt S-8c)D-\frac{1}{4}|\nabla\wt S|^2\\
			&\qquad+32\langle\ta,\ta_2\rangle\langle\tb,\ta_1\rangle-32\langle\ta,\ta_1\rangle\langle\tb,\ta_2\rangle.
		\end{aligned}
	\end{equation*}
	On the other hand, a direct calculation yields that
	\begin{equation*}
		8\left(\langle\ta,\ta_2\rangle+\langle\tb,\ta_1\rangle\right)^2+8\left(\langle\ta,\ta_1\rangle-\langle\tb,\ta_2\rangle\right)^2=\frac{1}{2}|\nabla\wt S|^2+32\langle\ta,\ta_2\rangle\langle\tb,\ta_1\rangle-32\langle\ta,\ta_1\rangle\langle\tb,\ta_2\rangle.
	\end{equation*}
	Combining Lemma \ref{simons} and the relation $2\wt S\tri\wt S=\tri(\wt S^2)-2|\nabla\wt S|^2$, we obtain that
	\begin{equation*}
		\begin{aligned}
			(\wt h_{ijk}^{\alpha}\tri\wt h_{ij}^{\alpha})_k&=\frac{1}{2}(2c-\wt S)\wt S^2+(2c-\wt S)(\wt S^2-|\wt A|^2-\rho^{\perp})+c\tri\wt S-\frac{3}{8}\tri(\wt S^2)\\
			&\qquad+8\left(\langle\ta,\ta_2\rangle+\langle\tb,\ta_1\rangle\right)^2+8\left(\langle\ta,\ta_1\rangle-\langle\tb,\ta_2\rangle\right)^2.
		\end{aligned}
	\end{equation*}
	Integrating over $M$, we obtain
	\begin{equation*}
		\int_M(2c-\wt S)D\ge\frac{1}{2}\int_M(2c-\wt S)\wt S^2.
	\end{equation*}
	Since the Gaussian curvature is positive, $2c-\wt S=2K>0$. By Lemma \ref{ineq}, we obtain
	\begin{equation*}
		D\le\frac{1}{2}\wt S^2.
	\end{equation*}
	Hence
	\begin{equation*}
		\int_M(2c-\wt S)D\le\frac{1}{2}\int_M(2c-\wt S)\wt S^2.
	\end{equation*}
	It follows that equality holds everywhere. By Lemma \ref{ineq} we conclude that $\langle\ta,\tb\rangle=0$ and $|\ta|=|\tb|$. Equivalently,
	\begin{equation*}
		D = \frac{1}{2}\wt S^2.
	\end{equation*}
	Substituting this into Lemma \ref{simons} yields $$\frac{1}{2}\tri\wt S=\B_1+2c\wt S-\frac{3}{2}\wt S^2=\B_1-\frac{1}{2}\wt S(3\wt S-4c),$$ which proves the theorem.
\end{proof}

\begin{corollary}\label{coro1gap}
	Suppose that $M$ is a closed surface immersed in $\Sn$ with parallel mean curvature vector and positive Gaussian curvature. Under the foregoing assumptions and notations, we have
	\begin{enumerate}
		\item\label{abrel} the relations of $\ta$ and $\tb$: $$\langle\ta,\tb\rangle=0 \text{ and } |\ta|^2=|\tb|^2=\frac{1}{4}\wt S;$$
		\item\label{lapab} the Laplacians of $\ta$ and $\tb$: $$\tri\ta=\frac{1}{2}(4c-3\wt S)\ta \text{ and } \tri\tb=\frac{1}{2}(4c-3\wt S)\tb;$$
		\item\label{sarho} the relations of $\wt S$, $|\wt A|^2$ and $\rho^\perp$: $$|\wt A|^2=\frac{1}{2}\wt S^2 \text{ and } \wt\rho^{\perp}=\wt S^2.$$
	\end{enumerate}
\end{corollary}
\begin{proof}
	From the proof of Theorem \ref{1stlap}, we obtain that $\langle a,b\rangle = 0$ and $|a|^2 = |b|^2 = \frac{1}{4}\wt S$. Substituting the relations into \eqref{lapa} and \eqref{lapb}, we obtain that $$\tri\ta=(2c-\wt S-2|\tb|^2)\ta=\frac{1}{2}(4c-3\wt S)\ta,$$ and similarly, $$\tri\tb=\frac{1}{2}(4c-3\wt S)\tb.$$ Finally,
	\begin{equation*}
		|\wt A|^2=4|\ta|^4+4|\tb|^4+8\langle\ta,\tb\rangle^2=8|\ta|^4=\frac{1}{2}\wt S^2,
	\end{equation*}
	and
	\begin{equation*}
		\wt\rho^{\perp}=16\left(|\ta|^2|\tb|^2-\langle\ta,\tb\rangle^2\right)=16|\ta|^4=\wt S^2,
	\end{equation*}
	which complete the proof.
\end{proof}

\begin{remark}
	Corollary \ref{coro1gap} \eqref{abrel} can also be proved by using the method of holomorphic differentials; see, for example, \cite{Guad83}. Moreover, the assumption of positive Gaussian curvature can be weakened to the assumption that $M$ is topologically a two-sphere. The proof follows from the same holomorphic differential argument as in our previous paper \cite{Ding25}.
\end{remark}

\begin{theorem}\label{1stint}
	Let $M$ be a closed surface immersed in $\Sn$ with parallel mean curvature vector and positive Gaussian curvature. Then $$\int_M\wt S(3\wt S-4c)=2\int_M\B_1\ge 0.$$ In particular, if $0\le\wt S\le\frac{4}{3}c$, then $\wt S\equiv 0$ or $\wt S\equiv\frac{4}{3}c$.
\end{theorem}
\begin{proof}
	Integrating Theorem \ref{1stlap} over $M$ and using $\B_1\ge0$, we obtain $$\int_M\wt S(3\wt S-4c)=2\int_M\B_1\ge0.$$ If $0\le\wt S\le\frac{4}{3}c$, then $$\wt S(3\wt S-4c)\le0.$$ Combining these two inequalities, we get $$\wt S(3\wt S-4c)=0,$$ and therefore $\wt S\equiv 0$ or $\wt S\equiv\frac{4}{3}c$.
\end{proof}

\section{The second gap}
In this section we establish the second Simons-type integral identity and prove the corresponding pinching theorem.

\begin{theorem}\label{2ndint}
	Let $M$ be a closed surface immersed in $\Sn$ with parallel mean curvature vector and positive Gaussian curvature. Then $$\int_M\wt S(3\wt S-4c)(3\wt S-5c)=2\int_M\left[\B_2-\frac{1}{4}\wt S(3\wt S-4c)^2+\frac{1}{2}|\nabla\wt S|^2\right]\ge0.$$ In particular, if $\frac{4}{3}c\le\wt S\le\frac{5}{3}c$, then $\wt S\equiv\frac{4}{3}c$ or $\wt S\equiv\frac{5}{3}c$.
\end{theorem}
\begin{proof}
	By Ricci's formula \eqref{ric1}, we have
	\begin{equation*}
		\begin{aligned}
			\wt h_{ijk}^{\alpha}\tri\wt h_{ijk}^{\alpha}&=(\wt h_{ijk}^{\alpha}\tri\wt h_{ij}^{\alpha})_k-\sum_{i,j,\alpha}(\tri\wt h_{ij}^{\alpha})^2+2\wt h_{ijk}^{\alpha}\wt h_{pj}^{\alpha}R_{pikmm}+\wt h_{ijk}^{\alpha}\wt h_{ijp}^{\alpha}R_{pmkm}\\
			&\qquad+4\wt h_{ijk}^{\alpha}\wt h_{pjm}^{\alpha}R_{pikm}+2\wt h_{ijk}^{\alpha}\wt h_{ijm}^{\beta}R_{\beta\alpha km}+\wt h_{ijk}^{\alpha}\wt h_{ij}^{\beta}R_{\beta\alpha kmm}.
		\end{aligned}
	\end{equation*}
	By Corollary \ref{coro1gap},
	\begin{equation*}
		\begin{aligned}
			\sum_{i,j,\alpha}(\tri\wt h_{ij}^{\alpha})^2&=2\sum_{\alpha}\left((\tri\ta^{\alpha})^2+(\tri\tb^{\alpha})^2\right)\\
			&=\frac{1}{4}\wt S(3\wt S-4c)^2.
		\end{aligned}
	\end{equation*}
	By \eqref{cur1}, we have
	\begin{equation*}
		2\wt h_{ijk}^{\alpha}\wt h_{pj}^{\alpha}R_{pikmm}=-\frac{1}{2}|\nabla\wt S|^2,
	\end{equation*}
	and
	\begin{equation*}
		4\wt h_{ijk}^{\alpha}\wt h_{pjm}^{\alpha}R_{pikm}+\wt h_{ijk}^{\alpha}\wt h_{ijp}^{\alpha}R_{pmkm}=\frac{5}{2}(2c-\wt S)\B_1.
	\end{equation*}
	Combining Corollary \ref{coro1gap} and \eqref{cur2}, we get
	\begin{equation*}
		\begin{aligned}
			2\wt h_{ijk}^{\alpha}\wt h_{ijm}^{\beta}R_{\beta\alpha km}&=16(\ta_1^{\alpha}\ta_2^{\beta}-\ta_2^{\alpha}\ta_1^{\beta})(\ta^{\beta}\tb^{\alpha}-\ta^{\alpha}\tb^{\beta})\\
			&=-\frac{1}{2}|\nabla\wt S|^2.
		\end{aligned}
	\end{equation*}
	Finally, by \eqref{cur3} and Corollary \ref{coro1gap} \eqref{abrel}, a direct computation gives that
	\begin{equation*}
		\wt h_{ijk}^{\alpha}\wt h_{ij}^{\beta}R_{\beta\alpha kmm}=-\frac{1}{2}\wt S\B_1.
	\end{equation*}
	Hence, we obtain that
	\begin{equation*}
		\wt h_{ijk}^{\alpha}\tri\wt h_{ijk}^{\alpha}=(\wt h_{ijk}^{\alpha}\tri\wt h_{ij}^{\alpha})_k+(5c-3\wt S)\B_1-|\nabla\wt S|^2-\frac{1}{4}\wt S(3\wt S-4c)^2.
	\end{equation*}
	Next, we have the representation
	\begin{equation*}
		\begin{aligned}
			\B_2&=4\sum_{\alpha}\left((\ta_{11}^{\alpha})^2+(\ta_{22}^{\alpha})^2+(\ta_{12}^{\alpha})^2+(\ta_{21}^{\alpha})^2\right)\\
			&=2\sum_{\alpha}\left((\ta_{11}^{\alpha}+\ta_{22}^{\alpha})^2+(\ta_{12}^{\alpha}-\ta_{21}^{\alpha})^2\right)+2\sum_{\alpha}\left((\ta_{11}^{\alpha}-\ta_{22}^{\alpha})^2+(\ta_{12}^{\alpha}+\ta_{21}^{\alpha})^2\right)\\
			&=2\sum_{\alpha}\left((\tri\ta^{\alpha})^2+(\tri\tb^{\alpha})^2\right)+\C_1\\
			&=\frac{1}{4}\wt S(3\wt S-4c)^2+\C_1,
		\end{aligned}
	\end{equation*}
	where
	\begin{equation*}
		\begin{aligned}
			\C_1&=2\sum_{\alpha}\left((\ta_{11}^{\alpha}-\ta_{22}^{\alpha})^2+(\ta_{12}^{\alpha}+\ta_{21}^{\alpha})^2\right)\\
			&=\B_2-\frac{1}{4}\wt S(3\wt S-4c)^2\\
			&\ge0.
		\end{aligned}
	\end{equation*}
	Hence we conclude that
	\begin{equation*}
		\frac{1}{2}\tri\B_1=(\wt h_{ijk}^{\alpha}\tri\wt h_{ij}^{\alpha})_k+(5c-3\wt S)\B_1-|\nabla\wt S|^2+\C_1.
	\end{equation*}
	Combining Theorem \ref{1stlap} and $2\wt S\tri\wt S=\tri(\wt S^2)-2|\nabla\wt S|^2$, we obtain that
	\begin{equation*}
		\frac{1}{2}\tri\B_1=(\wt h_{ijk}^{\alpha}\tri\wt h_{ij}^{\alpha})_k+\frac{5}{2}c\tri\wt S-\frac{3}{4}\tri(\wt S^2)+\frac{1}{2}|\nabla\wt S|^2-\frac{1}{2}\wt S(3\wt S-4c)(3\wt S-5c)+\C_1.
	\end{equation*}
	Integrating both sides over $M$, we arrive at
	\begin{equation*}
		\begin{aligned}
			\int_M\wt S(3\wt S-4c)(3\wt S-5c)&=\int_M\left(2\C_1+|\nabla\wt S|^2\right)\\
			&=2\int_M\left[\B_2-\frac{1}{4}\wt S(3\wt S-4c)^2+\frac{1}{2}|\nabla\wt S|^2\right]\\
			&\ge0.
		\end{aligned}
	\end{equation*}
	If $\frac{4}{3}c\le\wt S\le\frac{5}{3}c$, then $\wt S(3\wt S-4c)(3\wt S-5c)\le0$. Combining with the above integral identity, we obtain that $$\wt S(3\wt S-4c)(3\wt S-5c)=0,$$ and therefore $\wt S\equiv\frac{4}{3}c$ or $\wt S\equiv\frac{5}{3}c$.
\end{proof}

\section{The third gap}
Before deriving the third Simons-type identities, we collect two auxiliary formulas for the traceless second fundamental form.

\begin{lemma}\label{eq:pmc-s3-a1-a2-relations}
	Let $M$ be a closed surface immersed in $\Sn$ with parallel mean curvature vector and positive Gaussian curvature. Then we have
	\begin{equation*}
		\begin{aligned}
			\langle\ta_1,\ta_2\rangle=0,&\quad|\ta_1|^2=|\ta_2|^2=\frac{1}{8}\B_1,\\
			\langle\ta,\ta_1\rangle=\langle\tb,\ta_2\rangle=\frac{1}{8}\wt S_1,&\quad\langle\ta,\ta_2\rangle=-\langle\tb,\ta_1\rangle=\frac{1}{8}\wt S_2,\\
			\langle \ta,\ta_{11}\rangle=\langle \tb,\ta_{21}\rangle=\frac{1}{8}(\wt S_{11}-\B_1),&\quad\langle\ta,\ta_{22}\rangle=-\langle\tb,\ta_{12}\rangle=\frac{1}{8}(\wt S_{22}-\B_1),\\
			\langle\ta,\ta_{12}\rangle=\langle \tb,\ta_{22}\rangle=\frac{1}{8}\wt S_{12},&\quad\langle \ta,\ta_{21}\rangle=-\langle \tb,\ta_{11}\rangle=\frac{1}{8}\wt S_{21},\\
			\langle \ta_1,\ta_{21}\rangle=-\langle \ta_2,\ta_{11}\rangle,&\quad\langle \ta_1,\ta_{22}\rangle=-\langle \ta_2,\ta_{12}\rangle,\\
			\langle \ta_1,\ta_{11}\rangle=\langle \ta_2,\ta_{21}\rangle=\frac{1}{16}(\B_1)_1,&\quad\langle \ta_1,\ta_{12}\rangle=\langle \ta_2,\ta_{22}\rangle=\frac{1}{16}(\B_1)_2.
		\end{aligned}
	\end{equation*}
\end{lemma}

\begin{proof}
	Define $$\phi=\left(|\wt a_1|^2-|\wt a_2|^2-2\sqrt{-1}\langle \wt a_1,\wt a_2\rangle\right)dz^6,$$ which is independent of the choice of local coordinates. We show that it is holomorphic. Indeed, using \eqref{cod} and Corollary \ref{coro1gap} \eqref{lapab}, we have
	\begin{equation*}
		\begin{aligned}
			e_1\left(|\wt a_1|^2-|\wt a_2|^2\right)+e_2\left(2\langle\wt a_1,\wt a_2\rangle\right)&=2\langle\wt a_1,\wt a_{11}\rangle-2\langle\wt a_2,\wt a_{21}\rangle+2\langle\wt a_{12},\wt a_2\rangle+2\langle\wt a_1,\wt a_{22}\rangle\\
			&=2\langle\wt a_1,\triangle\wt a\rangle-2\langle\wt a_2,\triangle\wt b\rangle\\
			&=0.
		\end{aligned}
	\end{equation*}
	Similarly, $$e_2\left(|\wt a_1|^2-|\wt a_2|^2\right)-e_1\left(2\langle\wt a_1,\wt a_2\rangle\right)=0.$$ Thus $\phi$ is holomorphic. Since $\dim M=2$, every holomorphic differential of degree $4$ vanishes identically. Together with $\B_1=4\left(|\wt a_1|^2+|\wt a_2|^2\right)$, we obtain $$\langle\wt a_1,\wt a_2\rangle=0\text{ and }|\wt a_1|^2=|\wt a_2|^2=\frac18\B_1.$$ Differentiating Corollary \ref{coro1gap} \eqref{abrel}, we obtain $$\langle\wt a,\wt a_1\rangle=\langle\wt b,\wt a_2\rangle=\frac18\wt S_1\text{ and }\langle\wt a,\wt a_2\rangle=-\langle\wt b,\wt a_1\rangle=\frac18\wt S_2.$$ Differentiating once more, we obtain
	\begin{equation*}
		\begin{aligned}
			\langle\wt a,\wt a_{11}\rangle&=\langle\wt b,\wt a_{21}\rangle=\frac18(\wt S_{11}-\B_1),\\
			\langle\wt a,\wt a_{22}\rangle&=-\langle\wt b,\wt a_{12}\rangle=\frac18(\wt S_{22}-\B_1),\\
			\langle\wt a,\wt a_{12}\rangle&=\langle\wt b,\wt a_{22}\rangle=\frac18\wt S_{12},\\
			\langle\wt a,\wt a_{21}\rangle&=-\langle\wt b,\wt a_{11}\rangle=\frac18\wt S_{21}.
		\end{aligned}
	\end{equation*}
	Finally, differentiating $\langle\wt a_1,\wt a_2\rangle=0$ and $|\wt a_1|^2=|\wt a_2|^2=\frac18\B_1$, we obtain
	\begin{equation*}
		\begin{aligned}
			\langle\wt a_1,\wt a_{21}\rangle=-\langle\wt a_2,\wt a_{11}\rangle,&\quad\langle\wt a_1,\wt a_{22}\rangle=-\langle\wt a_2,\wt a_{12}\rangle,\\
			\langle\wt a_1,\wt a_{11}\rangle=\langle\wt a_2,\wt a_{21}\rangle=\frac1{16}(\B_1)_1,&\quad\langle\wt a_1,\wt a_{12}\rangle=\langle\wt a_2,\wt a_{22}\rangle=\frac1{16}(\B_1)_2.
		\end{aligned}
	\end{equation*}
	The proof is complete.
\end{proof}

\begin{lemma}\label{eq:pmc-s3-half-delta-B1}
	Let $M$ be a closed surface immersed in $\Sn$ with parallel mean curvature vector and positive Gaussian curvature. Then we have
	\begin{equation*}
		\frac{1}{2}\tri\B_1=\frac{7}{2}c\tri\wt S-\frac{9}{8}\tri(\wt S^2)+\frac{1}{2}|\nabla\wt S|^2-\frac{1}{4}\wt S(3\wt S-4c)(9\wt S-14c)+\B_2.
	\end{equation*}
\end{lemma}
\begin{proof}
	By the proof of Theorem \ref{2ndint}, we have
	\begin{equation*}
		\thh^\alpha_{ijk}\tri\thh^\alpha_{ijk}=(\thh^\alpha_{ijk}\tri\thh^\alpha_{ij})_k+(5c-3\tS)B_1-|\nabla\tS|^2-\frac14\tS(3\tS-4c)^2.
	\end{equation*}
	Hence
	\begin{equation*}
		\frac12\tri\B_1=(\thh^\alpha_{ijk}\tri\thh^\alpha_{ij})_k+(5c-3\tS)\B_1-|\nabla\tS|^2-\frac14\tS(3\tS-4c)^2+\B_2 .
	\end{equation*}
	By Lemma \ref{eq:pmc-s3-a1-a2-relations} and the proof of Theorem \ref{1stlap}, we have
	\begin{equation*}
		\begin{aligned}
			(\wt h_{ijk}^{\alpha}\tri\wt h_{ij}^{\alpha})_k&=\frac{1}{2}(2c-\wt S)\wt S^2+(2c-\wt S)(\wt S^2-|\wt A|^2-\rho^{\perp})+c\tri\wt S-\frac{3}{8}\tri(\wt S^2)\\
			&\qquad+8\left(\langle\ta,\ta_2\rangle+\langle\tb,\ta_1\rangle\right)^2+8\left(\langle\ta,\ta_1\rangle-\langle\tb,\ta_2\rangle\right)^2\\
			&=c\tri\tS-\frac38\tri(\tS^2).
		\end{aligned}
	\end{equation*}
	Therefore, we obtain the lemma.
\end{proof}

\subsection{The decomposition of $\B_3$}
We next compute $\B_3=|\nabla^3\thh|^2$.
\begin{lemma}\label{lemma:pmc-s3-B3}
	Let $M$ be a closed surface immersed in $\Sn$ with parallel mean curvature vector and positive Gaussian curvature. Then we have
	\begin{equation*}
		\B_3=\frac14(45\tS^2-144c\tS+116c^2)\B_1+\frac{13}{8}(7\tS-8c)|\nabla\tS|^2+\C_2+\C_3,
	\end{equation*}
	where
	\begin{equation*}
		\C_2=2\sum_\alpha\left[(\ta^\alpha_{111}-\ta^\alpha_{122})^2+(\ta^\alpha_{211}-\ta^\alpha_{222})^2\right],
	\end{equation*}
	and
	\begin{equation*}
		\C_3=2\sum_\alpha\left[(\ta^\alpha_{112}+\ta^\alpha_{121})^2+(\ta^\alpha_{212}+\ta^\alpha_{221})^2\right].
	\end{equation*}
\end{lemma}

\begin{proof}
	First, by Corollary \ref{coro1gap} \eqref{lapab} and Ricci's formula \eqref{ric1} we obtain
	\begin{equation}
		\begin{aligned}
	        \tri\ta_1&=\frac12(14c-9\tS)\ta_1+\frac74(-\ta\tS_1+\tb\tS_2),\\
	        \tri\ta_2&=\frac12(14c-9\tS)\ta_2-\frac74(\tb\tS_1+\ta\tS_2).
	       \end{aligned}
	\end{equation}
	Combining Lemma \ref{eq:pmc-s3-a1-a2-relations}, a direct expansion yields that
	\begin{equation}\label{eq:pmc-s3-lap-a1-a2-square}
		|\tri\ta_1|^2+|\tri\ta_2|^2=\frac1{16}(9\tS-14c)^2\B_1
	        +\frac7{32}(25\tS-28c)|\nabla\tS|^2 .
	\end{equation}
	Thus
	\begin{equation*}
	\begin{aligned}
	&\hspace{1.3em}4\sum_\alpha\left((\ta^\alpha_{111})^2+(\ta^\alpha_{122})^2+(\ta^\alpha_{211})^2+(\ta^\alpha_{222})^2\right)\\
	&=2\sum_\alpha\left((\ta^\alpha_{111}+\ta^\alpha_{122})^2+(\ta^\alpha_{211}+\ta^\alpha_{222})^2\right)+\C_2  \\
	&=2(|\tri\ta_1|^2+|\tri\ta_2|^2)+\C_2\\
	&=\frac18(9\tS-14c)^2\B_1+\frac7{16}(25\tS-28c)|\nabla\tS|^2+\C_2,
	\end{aligned}
	\end{equation*}
	where
	\begin{equation*}\label{eq:pmc-s3-C2-def}
	\C_2=2\sum_\alpha\left((\ta^\alpha_{111}-\ta^\alpha_{122})^2+(\ta^\alpha_{211}-\ta^\alpha_{222})^2\right).
	\end{equation*}
	For the remaining part, \eqref{cur2} and Ricci's formula \eqref{ric2} give that
	\begin{equation*}
		\begin{aligned}
			&\hspace{1.3em}4\sum_\alpha\left((\ta^\alpha_{112})^2+(\ta^\alpha_{121})^2+(\ta^\alpha_{212})^2+(\ta^\alpha_{221})^2\right)\\
			&=2\sum_\alpha\left((\ta^\alpha_{112}-\ta^\alpha_{121})^2+(\ta^\alpha_{212}-\ta^\alpha_{221})^2\right)+\C_3\\
			&=2\sum_\alpha\left[\left(3\ta^\alpha_2\left(\frac12\tS-c\right)+\ta^\beta_1R_{\beta\alpha12}\right)^2+\left(3\ta^\alpha_1\left(c-\frac12\tS\right)+\ta^\beta_2R_{\beta\alpha12}\right)^2\right]+\C_3\\
			&=\frac92\left(c-\frac12\tS\right)^2\B_1+\frac1{16}(7\tS-12c)|\nabla\tS|^2+\C_3,
		\end{aligned}
	\end{equation*}
	where
	\begin{equation*}
		\C_3=2\sum_\alpha\left((\ta^\alpha_{112}+\ta^\alpha_{121})^2+(\ta^\alpha_{212}+\ta^\alpha_{221})^2\right).
	\end{equation*}
	Therefore, we obtain that
	\begin{equation*}
		\B_3=\frac14(45\tS^2-144c\tS+116c^2)\B_1+\frac{13}{8}(7\tS-8c)|\nabla\tS|^2+\C_2+\C_3,
	\end{equation*}
	which proves the lemma.
\end{proof}

\subsection{The Laplacian of $\B_2$}
We now compute $\tri\B_2$.
\begin{lemma}
	Let $M$ be a closed surface immersed in $\Sn$ with parallel mean curvature vector and positive Gaussian curvature. Then we have
	\begin{equation*}
		\begin{aligned}
			\frac12\tri\B_2&=(\thh^\alpha_{ijkl}\tri\thh^\alpha_{ijk})_l-(21\tS^2-64c\tS+49c^2)\B_1+\frac14\tS(3\tS-4c)^2(7\tS-12c)+\frac14(\tri\tS)^2\\
			&\hspace{1.3em}+7\left(c-\frac12\tS\right)\B_2-\frac72(7\tS-8c)|\nabla\tS|^2-\langle\nabla\B_1,\nabla\tS\rangle-\frac12|\Hess\tS|^2+\B_3.
		\end{aligned}
	\end{equation*}
\end{lemma}
\begin{proof}
Ricci's formulas \eqref{ric2} and \eqref{ric3} give that
\begin{equation*}
\begin{aligned}
\frac12\tri\B_2&=(\thh^\alpha_{ijkl}\tri\thh^\alpha_{ijk})_l-\sum(\tri\thh^\alpha_{ijk})^2+6\thh^\alpha_{ijkl}\thh^\alpha_{pjkm}R_{pilm}+\thh^\alpha_{ijkl}\thh^\alpha_{ijkp}R_{pmlm}\\
&+3\thh^\alpha_{ijkl}\thh^\alpha_{pjk}R_{pilmm}+\thh^\alpha_{ijkl}\thh^\beta_{ijk}R_{\beta\alpha lmm}+2\thh^\alpha_{ijkl}\thh^\beta_{ijkm}R_{\beta\alpha lm}+\B_3.
\end{aligned}
\end{equation*}
Each term on the right-hand side is evaluated as follows:
\begin{equation*}
\begin{aligned}
-\sum(\tri\thh^\alpha_{ijk})^2&=-\frac14(9\tS-14c)^2\B_1-\frac78(25\tS-28c)|\nabla\tS|^2.\\
6\thh^\alpha_{ijkl}\thh^\alpha_{pjkm}R_{pilm}+\thh^\alpha_{ijkl}\thh^\alpha_{ijkp}R_{pmlm}&=7\left(c-\frac12\tS\right)\B_2-3\tS\left(c-\frac12\tS\right)(3\tS-4c)^2.\\
3\thh^\alpha_{ijkl}\thh^\alpha_{pjk}R_{pilmm}&=-\frac34\langle\nabla\B_1,\nabla\tS\rangle-\frac34(3\tS-4c)|\nabla\tS|^2.\\
\thh^\alpha_{ijkl}\thh^\beta_{ijk}R_{\beta\alpha lmm}&=-\frac14\langle\nabla\B_1,\nabla\tS\rangle-\frac18(3\tS-4c)|\nabla\tS|^2-\frac14\tS(3\tS-4c)\B_1.\\
2\thh^\alpha_{ijkl}\thh^\beta_{ijkm}R_{\beta\alpha lm}&=\frac14\tS^2(3\tS-4c)^2+\frac14(\tri\tS)^2-\frac12|\Hess\tS|^2.
\end{aligned}
\end{equation*}
Therefore, we get
\begin{equation*}
\begin{aligned}
\frac12\tri\B_2&=(\thh^\alpha_{ijkl}\tri\thh^\alpha_{ijk})_l-(21\tS^2-64c\tS+49c^2)\B_1+\frac14\tS(3\tS-4c)^2(7\tS-12c)\\
&+7\left(c-\frac12\tS\right)\B_2-\frac72(7\tS-8c)|\nabla\tS|^2-\langle\nabla\B_1,\nabla\tS\rangle+\frac14(\tri\tS)^2-\frac12|\Hess\tS|^2+\B_3,
\end{aligned}
\end{equation*}
which proves the lemma.
\end{proof}

\subsection{The third Simons-type integral identity}
We now derive the integral identity used for the third gap.
\begin{theorem}\label{3rdint}
	Let $M$ be a closed surface immersed in $\Sn$ with parallel mean curvature vector and positive Gaussian curvature. Then we have
	\begin{equation*}
		\begin{aligned}
			&\hspace{1.3em}\int_M\tS(3\tS-4c)(3\tS-5c)(5\tS-9c)\\
			&=\int_M\left[\frac32(11\tS-21c)|\nabla\tS|^2-\frac54(\triangle\tS)^2+2\C_2+2\C_3\right]\\
			&=2\int_M\left[\B_3-\frac18\tS(3\tS-4c)(45\tS^2-144c\tS+116c^2)+\frac18(65\tS-166c)|\nabla\tS|^2-\frac58(\triangle\tS)^2\right],
		\end{aligned}
	\end{equation*}
	where $\C_2=2\sum\left[(\ta_{111}^\alpha-\ta_{122}^\alpha)^2+(\ta_{211}^\alpha-\ta_{222}^\alpha)^2\right]$, $\C_3=2\sum\left[(\ta_{112}^\alpha+\ta_{121}^\alpha)^2+(\ta_{212}^\alpha+\ta_{221}^\alpha)^2\right]$.
\end{theorem}

\begin{proof}
	Combining Theorem \ref{1stlap}, Lemma \ref{eq:pmc-s3-half-delta-B1} and Lemma \ref{lemma:pmc-s3-B3}, we obtain
	\begin{equation*}
		\begin{aligned}
			\frac12\tri\B_2&=(\thh^\alpha_{ijkl}\tri\thh^\alpha_{ijk})_l-\frac12\wt S(3\wt S-4c)(3\wt S-5c)(5\wt S-9c)-\frac18(91\wt S-92c)|\nabla\wt S|^2\\
			&\hspace{1.3em}-\frac{39}{8}\wt S^2\tri\wt S+\frac{119}{4}c\wt S\tri\wt S+\frac{105}{8}c\tri(\wt S^2)-\frac{105}{16}\wt S\tri(\wt S^2)-\frac{83}{2}c^2\tri\wt S\\
			&\hspace{1.3em}+\frac74c\tri^2\wt S-\frac78\wt S\tri^2\wt S-\langle\nabla\B_1,\nabla \wt S\rangle+\frac14(\tri\wt S)^2-\frac12|\Hess \wt S|^2+\C_2+\C_3.
		\end{aligned}
	\end{equation*}
	Integrating on both sides of this equality, using the elementary identities
	\begin{equation*}
		\begin{aligned}
			\wt S\tri\wt S&=\frac12\tri(\wt S^2)-|\nabla\wt S|^2,\\
			\wt S\tri(\wt S^2)&=\frac23\tri(\wt S^3)-2\wt S|\nabla \wt S|^2,\\
			\wt S^2\tri\wt S&=\frac13\tri(\wt S^3)-2\wt S|\nabla \wt S|^2,\\
			\wt S\tri^2\wt S&=\tri(\wt S\tri\wt S)-2\operatorname{div}(\tri\wt S\nabla\wt S)+(\tri\wt S)^2,
		\end{aligned}
	\end{equation*}
	and
	\begin{equation*}
		\begin{aligned}
			-\langle\nabla\B_1,\nabla\wt S\rangle&=-\operatorname{div}(\B_1\nabla\wt S)+\B_1\tri\wt S\\
			&=-\operatorname{div}(\B_1\nabla \wt S)+\frac12(\tri\wt S)^2+\frac32\wt S^2\tri\wt S-2c\wt S\tri\wt S\\
			&=-\operatorname{div}(\B_1\nabla \wt S)+\frac12(\tri\wt S)^2+\frac12\tri(\wt S^3)-3\wt S|\nabla\wt S|^2-c\tri(\wt S^2)+2c|\nabla\wt S|^2,
		\end{aligned}
	\end{equation*}
	we obtain that
	\begin{equation*}
		\begin{aligned}
			\int_M\frac12\wt S(3\wt S-4c)(3\wt S-5c)(5\wt S-9c)=\int_M\left[\frac14(34\wt S-65c)|\nabla\wt S|^2-\frac18(\tri\wt S)^2-\frac12|\Hess\wt S|^2+\C_2+\C_3\right].
		\end{aligned}
	\end{equation*}
	On the other hand, Reilly's formula gives that
	\begin{equation*}
		\int_M\left[2(\tri\wt S)^2-2|\Hess\wt S|^2+(\wt S-2c)|\nabla\wt S|^2\right]=0.
	\end{equation*}
	Therefore, we obtain
	\begin{equation*}
		\begin{aligned}
			&\hspace{1.3em}\int_M\tS(3\tS-4c)(3\tS-5c)(5\tS-9c)\\
			&=\int_M\left[\frac32(11\tS-21c)|\nabla\tS|^2-\frac54(\triangle\tS)^2+2\C_2+2\C_3\right]\\
			&=2\int_M\left[\B_3-\frac18\tS(3\tS-4c)(45\tS^2-144c\tS+116c^2)+\frac18(65\tS-166c)|\nabla\tS|^2-\frac58(\triangle\tS)^2\right],
		\end{aligned}
	\end{equation*}
	which proves the theorem.
\end{proof}

\subsection{The lower bound for $\C_2+\C_3$}
We now establish the lower bound for $\C_2+\C_3$.
\begin{lemma}\label{infc2c3}
	Let $M$ be a closed surface immersed in $\Sn$ with parallel mean curvature vector and positive Gaussian curvature. Then we have
	\begin{equation*}
		\C_2+\C_3\ge\frac{9}{8}\tS|\nabla\tS|^2.
	\end{equation*}
\end{lemma}
\begin{proof}
	Define
	\begin{equation*}
		\begin{aligned}
			\tE_1&=\ta_{111}-\ta_{122}+\ta_{212}+\ta_{221},\\
			\tE_2&=\ta_{112}+\ta_{121}+\ta_{222}-\ta_{211}.
		\end{aligned}
	\end{equation*}
	By Corollary \ref{coro1gap} \eqref{lapab}, we have
	\begin{align*}
		(\ta_{11}+\ta_{22})_1&=\frac12(4c-3\tS)\ta_1-\frac32\ta\tS_1,\\
		(\ta_{11}+\ta_{22})_2&=\frac12(4c-3\tS)\ta_2-\frac32\ta\tS_2,\\
		(\ta_{21}-\ta_{12})_1&=\frac12(4c-3\tS)\ta_2-\frac32\tb\tS_1,\\
		(\ta_{21}-\ta_{12})_2&=-\frac12(4c-3\tS)\ta_1-\frac32\tb\tS_2.
	\end{align*}
	Hence
	\begin{equation*}
		\begin{aligned}
			\tE_1&=(\ta_{11}+\ta_{22})_1+(\ta_{21}-\ta_{12})_2=-\frac32(\ta\tS_1+\tb\tS_2),\\
			\tE_2&=(\ta_{11}+\ta_{22})_2-(\ta_{21}-\ta_{12})_1=\frac32(\tb\tS_1-\ta\tS_2).
		\end{aligned}
	\end{equation*}
	Therefore, we get
	\begin{equation*}
		|\tE_1|^2+|\tE_2|^2=\frac98\tS|\nabla\tS|^2.
	\end{equation*}
	Combining Cauchy's inequality, we obtain
	\begin{equation*}
		\C_2+\C_3\ge|\tE_1|^2+|\tE_2|^2\ge\frac98\tS|\nabla\tS|^2,
	\end{equation*}
	which proves the lemma.
\end{proof}

Combining Theorem \ref{3rdint} and Lemma \ref{infc2c3}, we obtain the following theorem.
\begin{theorem}\label{new3rdint}
	Let $M$ be a closed surface immersed in $\Sn$ with parallel mean curvature vector and positive Gaussian curvature. Then we have
	\begin{equation*}
		\begin{aligned}
			\int_M\tS(3\tS-4c)(3\tS-5c)(5\tS-9c)\ge\int_M\left[\frac34(25\tS-42c)|\nabla\tS|^2-\frac54(\tri\tS)^2\right].
		\end{aligned}
	\end{equation*}
\end{theorem}

\subsection{The parameter estimate}
Set $\tS_{\min}=\inf_M\tS\text{ and }\tS_{\max}=\sup_M\tS$. From
\begin{lemma}\label{estimatelem}
	Let $M$ be a closed surface immersed in $\Sn$ with parallel mean curvature vector and positive Gaussian curvature. Then we have
	\begin{equation*}
		\begin{aligned}
			&\hspace{1.3em}\int_M\tS(3\tS-4c)(3\tS-5c)(5\tS-9c)\\
			&\ge\int_M\left[
			-\frac{5(w-\tS)^2}{16t(1-t)}
			\left(\frac{2+15t}{2}(w+\tS)+\frac{36}{5}c
			 -2\tS_{\max}-\frac{126}{5}ct\right)^2-\frac{5\tS(\tS_{\max}-\tS)(3\tS-4c)(3\tS-5c)}{2(1-t)}\right],
		\end{aligned}
	\end{equation*}
	for all $0<t<1$ and $w\in\mathbb{R}$.
\end{lemma}

\begin{proof}
	First we have
	\begin{equation*}
		\begin{aligned}
			\int_M(\tri\tS)^2&=2\int_M\B_1\tri\tS-\int_M\tS(3\tS-4c)\tri\tS\\
			&=2\int_M\B_1\tri\tS+\int_M(6\tS-4c)|\nabla\tS|^2.
		\end{aligned}
	\end{equation*}
	Using \eqref{eq:pmc-s3-half-delta-B1}, integration by parts gives
	\begin{equation*}
		\begin{aligned}
			\frac12\int_M\B_1\tri\tS&=\frac12\int_M\tS\tri\B_1\\
			&=\int_M\left[\left(\frac{11}{4}\tS-\frac72c\right)|\nabla \tS|^2-\frac14\tS^2(3\tS-4c)(9\tS-14c)+\tS\B_2\right].
		\end{aligned}
	\end{equation*}
	The second Simons-type identity gives
	\begin{equation*}
		\int_M\C_1=\int_M\left[\frac12\tS(3\tS-4c)(3\tS-5c)-\frac12|\nabla\tS|^2\right].
	\end{equation*}
	Hence we get
	\begin{equation*}
		\int_M(\tri\tS)^2=\int_M\left[(17\tS-18c)|\nabla\tS|^2+4\tS\C_1-2\tS^2(3\tS-4c)(3\tS-5c)\right].
	\end{equation*}
	Thus, for any $0<t<1$,
	\begin{equation*}
		\begin{aligned}
			&\int_M\left[(\tri\tS)^2+\frac{3(1-t)}5(42c-25\tS)|\nabla\tS|^2\right]\\
			&\le\int_M\left[\left((2+15t)\tS+\frac{36}{5}c-2\tS_{\max}-\frac{126}{5}ct\right)|\nabla\tS|^2+2\tS(\tS_{\max}-\tS)(3\tS-4c)(3\tS-5c)\right]\\
			&=\int_M\left[(w-\tS)\left(\frac{2+15t}{2}(w+\tS)+\frac{36}{5}c-2\tS_{\max}-\frac{126}{5}ct\right)\tri\tS+2\tS(\tS_{\max}-\tS)(3\tS-4c)(3\tS-5c)\right].
		\end{aligned}
	\end{equation*}
	Thus, by Cauchy-Schwarz's inequality and Young's inequality we have
    \begin{equation*}
		\begin{aligned}
			&\hspace{1.3em}\int_M\left[(\tri\tS)^2+\frac{3(1-t)}{5}(42c-25\tS)|\nabla\tS|^2\right]\\
			&\le\left[\int_M(w-\tS)^2\left(\frac{2+15t}{2}(w+\tS)+\frac{36}{5}c-2\tS_{\max}-\frac{126}{5}ct\right)^2\right]^{\frac{1}{2}}\left[\int_M(\tri\tS)^2\right]^{\frac{1}{2}}\\
			&\hspace{3em}+\int_M2\tS(\tS_{\max}-\tS)(3\tS-4c)(3\tS-5c)\\
			&\le\frac{1}{4t}\int_M(w-\tS)^2\left(\frac{2+15t}{2}(w+\tS)+\frac{36}{5}c-2\tS_{\max}-\frac{126}{5}ct\right)^2+t\int_M(\tri\tS)^2\\
			&\hspace{3em}+\int_M2\tS(\tS_{\max}-\tS)(3\tS-4c)(3\tS-5c),
		\end{aligned}
	\end{equation*}
	which yields that, for $0<t<1$,
	\begin{equation*}
		\begin{aligned}
			&\hspace{1.3em}\int_M\left[(\tri\tS)^2+\frac35(42c-25\tS)|\nabla\tS|^2\right]\\
			&\le\frac{1}{4t(1-t)}\int_M(w-\tS)^2\left[\frac{2+15t}{2}(w+\tS)+\frac{36}{5}c-2\tS_{\max}-\frac{126}{5}ct\right]^2\\
			&\hspace{1.3em}+\frac{2}{1-t}\int_M\tS(\tS_{\max}-\tS)(3\tS-4c)(3\tS-5c).
		\end{aligned}
	\end{equation*}
	Combining \eqref{new3rdint} yields that, for all $0<t<1$ and $w\in\R$,
	\begin{equation*}
		\begin{aligned}
			&\hspace{1.3em}\int_M\tS(3\tS-4c)(3\tS-5c)(5\tS-9c)\\
			&\ge\int_M\left[-\frac{5(w-\tS)^2}{16t(1-t)}\left(\frac{2+15t}{2}(w+\tS)+\frac{36}{5}c-2\tS_{\max}-\frac{126}{5}ct\right)^2-\frac{5\tS(\tS_{\max}-\tS)(3\tS-4c)(3\tS-5c)}{2(1-t)}\right],
		\end{aligned}
	\end{equation*}
	which proves the lemma.
\end{proof}

\subsection{Endpoint gaps and oscillation estimate in the third interval}
We now assume
\begin{equation*}
	\frac53c\le\tS\le\frac95c.
\end{equation*}
Then $\tS<2c$ and $K>0$.

\begin{theorem}[Left endpoint gap]\label{s3-left-gap}
Let $M$ be a closed surface immersed in $\mathbb S^N$ with parallel mean curvature vector and positive Gaussian curvature. If $$\frac53c\le\tS<cS_-,$$ then $\tS\equiv\frac53c$.
\end{theorem}

\begin{proof}
Suppose that $\widetilde S\not\equiv\frac53c$. By Lemma \ref{estimatelem},
\begin{align*}
&\quad\int_M\tS(3\tS-4c)(3\tS-5c)\left(5\tS-9c+\frac{5(\tS_{\max}-\tS)}{2(1-t)}\right)\\
&\geq\int_M-\frac{5(w-\tS)^2}{16t(1-t)}\left[\frac{2+15t}{2}(w+\tS)+\frac{36}{5}c-2\tS_{\max}-\frac{126}{5}ct\right]^2.
\end{align*}
If $0<t\leq\frac12$, then
\begin{align*}
5\tS-9c+\frac{5(\tS_{\max}-\tS)}{2(1-t)}&=\frac{5\tS_{\max}}{2(1-t)}+\frac{5(1-2t)\tS}{2(1-t)}-9c\\
&\leq5\tS_{\max}-9c.
\end{align*}
For $$\frac53c\leq w\leq\tS\leq\frac95c,$$ using $\tS_{\max}>\frac53c$, we have
\begin{align*}
\frac{(w-\tS)^2}{(3\tS-4c)(3\tS-5c)}&=\frac19\left(\frac{4c-3w}{3\tS-4c}+1\right)\left(\frac{5c-3w}{3\tS-5c}+1\right)&\leq\frac{(w-\tS_{\max})^2}{(3\tS_{\max}-4c)(3\tS_{\max}-5c)}.
\end{align*}
Consequently,
\begin{equation}\label{eq:pmc-third-pointwise}
	\int_M\tS(3\tS-4c)(3\tS-5c)\left[5\tS_{\max}-9c+\frac{5(w-\tS_{\max})^2\left(\frac{2+15t}{2}(w+\tS)+\frac{36}{5}c-2\tS_{\max}-\frac{126}{5}ct\right)^2}{16t(1-t)\tS(3\tS_{\max}-4c)(3\tS_{\max}-5c)}\right]\geq0.
\end{equation}
We now optimize the parameter $t$ in \eqref{eq:pmc-third-pointwise} with $w=\frac53c$. Define $$\lambda_t\coloneqq1+\frac{15}{2}t,\quad\mu_t\coloneqq\frac{133}{15}c-2\tS_{\max}-\frac{127}{10}ct,$$ and $$g_t(S)\coloneqq\frac{(\lambda_tS+\mu_t)^2}{S}.$$ Then \eqref{eq:pmc-third-pointwise} becomes $$\int_M\tS(3\tS-4c)(3\tS-5c)\left[5\tS_{\max}-9c+\frac{K_{\tS_{\max}}}{t(1-t)}g_t(\tS)\right]\geq0,$$ where $$K_{\tS_{\max}}\coloneqq\frac{5(3\tS_{\max}-5c)^2}{144(3\tS_{\max}-4c)(3\tS_{\max}-5c)}>0.$$ We first consider $\frac49\leq t\leq\frac12$. For $\frac53c\leq S\leq\tS_{\max}\leq\frac95c$, we have $\mu_t<0$ and $\lambda_tS+\mu_t>0$. Therefore $$g_t'(S)=\frac{(\lambda_tS+\mu_t)(\lambda_tS-\mu_t)}{S^2}>0.$$ Thus $g_t$ is increasing on $\left[\frac53c,\tS_{\max}\right]$, and hence
\begin{equation}\label{eqKSmax}
	5\tS_{\max}-9c+\frac{K_{\tS_{\max}}}{t(1-t)}g_t(\tS_{\max})\geq0.
\end{equation}
Put $\eta_1\coloneqq\frac{133}{15}c-\tS_{\max}$ and $\eta_2\coloneqq\frac{15}{2}\tS_{\max}-\frac{127}{10}c$.
Since $$\lambda_t\tS_{\max}+\mu_t=\eta_1+\eta_2t,$$ we obtain
\begin{equation*}\Theta_t(\tS_{\max})\coloneqq\tS_{\max}(3\tS_{\max}-4c)(5\tS_{\max}-9c)+\frac{5(3\tS_{\max}-5c)}{144t(1-t)}(\eta_1+\eta_2t)^2\geq0.
\end{equation*}
For fixed $\tS_{\max}$, it remains to minimize $$\phi(t)\coloneqq\frac{(\eta_1+\eta_2t)^2}{t(1-t)}.$$ Since $\eta_1+\eta_2t>0$ for $\frac53c\leq\widetilde S_{\max}\leq\frac95c$ and $0<t\leq\frac12$, a direct calculation gives $$\phi'(t)=\frac{(\eta_1+\eta_2t)\left((2\eta_1+\eta_2)t-\eta_1\right)}{t^2(1-t)^2}.$$ The unique critical point is $$\tau\coloneqq\frac{\eta_1}{2\eta_1+\eta_2}=\frac{2(133c-15\tS_{\max})}{151c+165\tS_{\max}}.$$ It follows that the minimum of $\phi$ on $(0,\frac12]$ is attained at
\begin{equation*}
t_{\tS_{\max}}
\coloneqq
\begin{cases}
\dfrac12,
&\dfrac53c<\tS_{\max}\leq\dfrac{127}{75}c,\\[6pt]
\dfrac{2(133c-15\tS_{\max})}{151c+165\tS_{\max}},
&\dfrac{127}{75}c<\tS_{\max}\leq\dfrac95c.
\end{cases}
\end{equation*}
Moreover, $\frac{53}{112}\leq t_{\tS_{\max}}\le\frac12$, and hence $t_{\tS_{\max}}\in\left[\frac49,\frac12\right]$. It remains to show that the range $0<t<\frac49$ cannot yield a better estimate. Since $$g_t''(s)=\frac{2\mu_t^2}{s^3}\geq0,$$ the function $g_t$ is convex, and hence its maximum on $\left[\frac53c,\tS_{\max}\right]$ is attained at an endpoint. For $0<t\leq\frac12$, define $$\mathcal E(t)\coloneqq\frac1{t(1-t)}\max_{\frac53c\leq S\leq\tS_{\max}}g_t(S).$$ Since $g_{t_{\tS_{\max}}}$ is increasing, we have $$\mathcal E(t_{\tS_{\max}})=\frac{\phi(t_{\tS_{\max}})}{\tS_{\max}}.$$ On the other hand, since $\frac{53}{112}>\frac49$, every $0<t<\frac49$ lies to the left of $t_{\tS_{\max}}$. As $\phi$ is decreasing on $(0,t_{\tS_{\max}}]$, we have $$\mathcal E(t)\ge\frac{\phi(t)}{\tS_{\max}}>\frac{\phi(t_{\tS_{\max}})}{\tS_{\max}}=\mathcal E(t_{\tS_{\max}}).$$ Thus the choice $t=t_{\tS_{\max}}$ is optimal among all $0<t\leq\frac12$ in the pointwise reduction \eqref{eq:pmc-third-pointwise}. If $\frac53c<\tS_{\max}\leq\frac{127}{75}c$, then $t_{\tS_{\max}}=\frac12$ and
\begin{equation*}
\Theta_{\frac12}(\tS_{\max})=\tS_{\max}(3\tS_{\max}-4c)(5\tS_{\max}-9c)+\frac5{36}(3\tS_{\max}-5c)\left(\frac{11}{4}\tS_{\max}+\frac{151}{60}c\right)^2.
\end{equation*}
A direct calculation gives $$\Theta_{\frac12}(\tS_{\max})<0$$ on this interval, which is a contradiction. If $\frac{127}{75}c<\tS_{\max}\leq\frac95c$, then $t_{\tS_{\max}}=\tau$. By the definition of $\tau$ we have $$\frac{(\eta_1+\eta_2\tau)^2}{\tau(1-\tau)}=4\eta_1(\eta_1+\eta_2).$$ Define $$\Theta_-(x)\coloneqq7965x^3-10935x^2-13509x+15295.$$ The polynomial $\Theta_-$ has the unique root $S_-$ in $\left(\frac53,\frac95\right)$, and a direct calculation gives $$\Theta_\tau(\tS_{\max})=\frac{c^3}{648}\Theta_-\left(\frac{\tS_{\max}}{c}\right).$$ If $\frac{127}{75}c<\tS_{\max}<cS_-$, then $\Theta_\tau(\tS_{\max})<0$, again a contradiction with \eqref{eqKSmax}. Consequently, the assumption $\tS\not\equiv\frac53c$ is impossible. Hence $\tS\equiv\frac53c$, which completes the proof.
\end{proof}

\begin{theorem}[Interior oscillation estimate]\label{s3-interior}
Let $M$ be a closed surface immersed in $\mathbb S^N$ with parallel mean curvature vector and positive Gaussian curvature. If $\frac53c\leq\tS\leq\frac95c$ and $\tS\not\equiv\frac53c$, then
\begin{equation*}
	\tS_{\max}-\tS_{\min}\ge cG\left(\frac{\tS_{\min}}{c}\right)=\frac{120\tS_{\min}(3\tS_{\min}-4c)(9c-5\tS_{\min})}{2225\tS_{\min}^2-420c\tS_{\min}-1701c^2}.
\end{equation*}
\end{theorem}

\begin{proof}
	Keeping the parameter $t$ free and arguing as in the proof of Theorem \ref{s3-left-gap}, for every $\frac53c\leq w\leq\tS_{\min}$ and $\frac49\leq t\leq\frac12$, we obtain $$0\leq w(3w-4c)(5\tS_{\max}-9c)+\frac{5(\tS_{\max}-w)}{48t(1-t)}\left(w+\frac{27}{5}c+\frac{75w-117c}{10}t\right)^2.$$ Since $5\tS_{\max}-9c=5(\tS_{\max}-w)-(9c-5w)$, it follows that $$\tS_{\max}-w\geq\Phi_t(w),$$ where $$\Phi_t(w)\coloneqq\frac{48t(1-t)w(3w-4c)(9c-5w)}{240t(1-t)w(3w-4c)+5\left(w+\frac{27}{5}c+\frac{75w-117c}{10}t\right)^2}.$$ For fixed $w$, differentiating $\Phi_t(w)$ with respect to $t$ shows that the maximum is attained at $$t=t_0(w)\coloneqq\frac{2(5w+27c)}{95w-9c}.$$ Moreover, $$t_0(w)\in\left[\frac49,\frac{53}{112}\right]\subset\left(0,\frac12\right)\quad\text{for }w\in\left[\frac53c,\frac95c\right].$$ The case $w=\frac95c$ is immediate since $\Phi_t\left(\frac95c\right)=0$. Substituting $t=t_0(w)$ gives $$\max_{\frac49\leq t\leq\frac12}\Phi_t(w)=\frac{120w(3w-4c)(9c-5w)}{2225w^2-420cw-1701c^2}.$$ Taking $w=\tS_{\min}$ proves the result.
\end{proof}

\begin{theorem}[Right endpoint gap]\label{s3-right-gap}
	Let $M$ be a closed surface immersed in $\Sn$ with parallel mean curvature vector and positive Gaussian curvature. If $$cS_+<\tS\leq\frac95c,$$ then $\tS\equiv\frac95c$.
\end{theorem}

\begin{proof}
	Assume that $\tS\not\equiv \frac95c$ and define $r_\beta\coloneqq\frac{1}{c}\tS_{\min}$. By Theorem \ref{new3rdint} and Lemma \ref{estimatelem}, for any $0<t\le \frac12$ and any constant $w$, we have
	\begin{equation*}
		\int_M\tS(3\tS-4c)(3\tS-5c)\left(5\tS-9c+\frac{5(\tS_{\max}-\tS)}{2(1-t)}\right)\ge-\frac{5\left(15t\tS_{\max}+\frac{36}{5}c-\frac{126}{5}ct\right)^2}{16t(1-t)}\int_M(\tS-w)^2.
	\end{equation*}
	Choose $w=\frac95c$. Then we get
	\begin{equation*}
		\int_M\tS(3\tS-4c)(3\tS-5c)\left(5\tS-9c+\frac{5(\tS_{\max}-\tS)}{2(1-t)}\right)\ge\int_M\frac{\left(15t\tS_{\max}+\frac{36}{5}c-\frac{126}{5}ct\right)^2}{-80t(1-t)}(5\tS-9c)^2.
	\end{equation*}
	Since $\tS_{\max}\le\frac95c$, for $0<t\le\frac12$, we have
	\begin{equation*}
		\begin{aligned}
			5\tS-9c+\frac{5(\tS_{\max}-\tS)}{2(1-t)}&=\frac{5\tS_{\max}+5(1-2t)\tS-18(1-t)c}{2(1-t)}\\
			&\le\frac{(1-2t)(5\tS-9c)}{2(1-t)} .
		\end{aligned}
	\end{equation*}
	It follows that
	\begin{equation*}
		\int_M(9c-5\tS)\left[\frac{2t-1}{2(1-t)}\tS(3\tS-4c)(3\tS-5c)+\frac{\left(15t\tS_{\max}+\frac{36}{5}c-\frac{126}{5}ct\right)^2}{80t(1-t)}(9c-5\tS)\right]\ge0.
	\end{equation*}
	By $\frac{5}{3}c<\tS\le\frac{9}{5}c$ and $\tS\not\equiv\frac{9}{5}c$, we have $\tS_{\min}<\frac{9}{5}c$ and
	\begin{equation*}
		\frac{2t-1}{2(1-t)}\tS_{\min}(3\tS_{\min}-4c)(3\tS_{\min}-5c)+\frac{\left(27ct+\frac{36}{5}c-\frac{126}{5}ct\right)^2}{80t(1-t)}(9c-5\tS_{\min})\ge0,
	\end{equation*}
	which yields $$40t(2t-1)r_\beta(3r_\beta-4)(3r_\beta-5)+\left(\frac95t+\frac365\right)^2(9-5r_\beta)\geq0.$$ To optimize the estimate near the upper endpoint, we maximize the coefficient $$\Lambda(t)\coloneqq\frac{40t(1-2t)}{\left(\frac95t+\frac365\right)^2},\quad0<t\leq\frac12.$$ A direct calculation gives that $\Lambda(t)$ attains its unique maximum at $t=\frac4{17}$. With this optimal choice, the preceding inequality becomes $$\Theta_+(r_\beta)\coloneqq1125r_\beta^3-3375r_\beta^2+9790r_\beta-13122\leq0.$$ The polynomial $\Theta_+$ has the unique root $S_+$ in $\left(\frac53,\frac95\right)$, and $$\Theta_+(r_\beta)>0\quad\text{for }r_\beta\in\left(S_+,\frac95\right].$$ Hence $r_\beta>S_+$ gives a contradiction. Therefore $\tS\equiv\frac95c$, which completes the proof.
\end{proof}

We next derive a small-oscillation rigidity consequence from the quantitative estimate obtained in Theorem \ref{main1}.

\begin{proof}[Proof of Theorem \ref{oscillation}]
Set $c=1+H^2$. Since $H$ is constant, we have $S_{\max}-S_{\min}=\tS_{\max}-\tS_{\min}$. If $\tS_{\max}<cS_-$, then Theorem \ref{s3-left-gap} applies. If $\tS_{\min}>cS_+$, then Theorem \ref{s3-right-gap} applies. Otherwise, $$\tS_{\min}\leq cS_+.$$ Since the case $\widetilde S\equiv\frac53c$ has already been covered, we may assume that $\tS\not\equiv\frac53c$. The function $G$ is decreasing on $\left[\frac53,\frac95\right]$. Therefore, by Theorem \ref{s3-interior},
\begin{equation*}
	S_{\max}-S_{\min}=\tS_{\max}-\tS_{\min}\ge cG\left(\frac{\tS_{\min}}{c}\right)\ge cG(S_+)=c\varepsilon_0=\varepsilon(1+H^2),
\end{equation*}
which contradicts the assumption. Hence $\tS\equiv\frac53c$ or $\tS\equiv\frac95c$. Since $S=\tS+2H^2$, it follows that $S\equiv\frac53+\frac{11}{3}H^2$ or $S\equiv\frac95+\frac{19}{5}H^2$.
\end{proof}

\subsection{Characterization of the endpoint cases via Yau's classification theorem}
Now we can give the proof of Proposition \ref{yauchara}.
\begin{proof}[Proof of Proposition \ref{yauchara}]
	If $H=0$, then $M$ is a minimal surface in $\Sn$, and the conclusion follows directly from Calabi's classification theorem. Hence we may assume $H\neq 0$. By Yau's classification theorem, a PMC surface in a space form is either contained in a $3$-dimensional totally umbilical submanifold, or is a minimal surface of a totally umbilical hypersurface.
	\begin{enumerate}
		\item In the first case, $M$ is contained in a $3$-dimensional umbilical submanifold $$Q^3\subset\Sn$$ and has constant mean curvature as a surface in $Q^3$. Since $Q^3$ is totally umbilical in $\Sn$, it is a $3$-dimensional space form of positive constant curvature. Moreover, the assumption $K>0$ implies that the orientable double cover of $M$ is a $2$-sphere. Hence, by the Hopf's theorem for constant mean curvature $2$-spheres in $3$-dimensional space forms \cite{Cher83}, $M$ is totally umbilical in $Q^3$. Since $Q^3$ is totally umbilical in $\Sn$, the composition formula for second fundamental forms implies that $M$ is totally umbilical in $\Sn$.  Therefore $$h(X,Y)=g(X,Y)\HH$$ for all $X,Y\in TM$, where $\HH$ is the mean curvature vector of $M\subset\Sn$.  Hence we obtain $$\wt h=h-\HH g=0.$$ Therefore we obtain $\wt S=0$.
		\item It remains to consider the second case. Then $M$ is minimal in an umbilical hypersurface $$Q^{N-1}\subset\Sn.$$ Let $B^Q$ be the second fundamental form of $Q^{N-1}$ in $\Sn$. Since $Q^{N-1}$ is totally umbilical, we may write $$B^Q(X,Y)=\mu g(X,Y)\xi$$ for some unit normal vector field $\xi$. Let $h^Q$ denote the second fundamental form of $M$ in $Q^{N-1}$. The composition formula gives $$h(X,Y)=h^Q(X,Y)+\mu g(X,Y)\xi.$$ Since $M$ is minimal in $Q^{N-1}$, we have $\operatorname{tr}h^Q=0$. Hence the mean curvature vector of $M$ in $\Sn$ is $\HH=\mu\xi$, which implies that $H^2=\mu^2$. The Gauss equation for the totally umbilical hypersurface $Q^{N-1}$ implies that $Q^{N-1}$ has sectional curvature $$\sec_Q=1+\mu^2=1+H^2=c.$$ Thus $Q^{N-1}$ is a round sphere of radius $\frac{1}{\sqrt{c}}$. Moreover, $$\wt h=h-\HH g=h^Q.$$ Consequently, $\wt S=|\wt h|^2=|h^Q|^2$. Thus $M$ is a minimal surface in a round sphere of sectional curvature $c$, and its second fundamental form in that sphere has squared norm $\wt S$. If $\wt S=cS(s)$, then by Gauss equation and the definition of $S(s)$ we have $$K=c-\frac{c}{2}S(s)=\frac{2c}{s(s+1)}.$$ By Calabi's classification theorem, $M$ is, up to a rigid motion, Calabi's $2$-sphere $$S^2\left(\frac{2c}{s(s+1)}\right)\to\mS^{2s}(c)\subset\Sn,\quad s=1,2,3,4,\cdots,$$ where the last inclusion is totally umbilical. For $s=1,2,3,4$, the corresponding values of the curvature $K$ are $$1+H^2,\quad\frac{1+H^2}{3},\quad\frac{1+H^2}{6},\quad\frac{1+H^2}{10},\quad\cdots,\quad\frac{2(1+H^2)}{s(s+1)},\quad\cdots.$$
	\end{enumerate}
	Therefore, we complete the proof.
\end{proof}

\par\bigskip
\noindent\textbf{Acknowledgement}

Jianquan Ge is partially supported by NSFC (No. 12571049) and the Fundamental Research Funds for the Central Universities. F. G. Li is partially supported by NSFC (No. 12271040 and 12501061), the Guangdong Provincial Association for Science and Technology Youth Talent Support Program (No. SKXRC2026413), and the Research Start-up Funding of Beijing Institute of Technology (No. 5640011253301).

\par\bigskip
\noindent\textbf{Data Availability Materials}

Data sharing is not applicable to this paper as no new data were created or analyzed in this study.

\par\bigskip
\noindent\textbf{Conflict of Interest}

The authors state that there is no conflict of interest.


\begin{thebibliography}{99}

\bibitem{Alen94} H. Alencar and M. do Carmo, {\em Hypersurfaces with constant mean curvature in spheres}, Proc. Amer. Math. Soc., {\bf 120} (1994), 1223--1229.

\bibitem{Benk79} K. Benko, M. Kothe, K. D. Semmler and U. Simon, {\em Eigenvalues of the Laplacian and curvature}, Colloq. Math., {\bf 42} (1979), 19--31.

\bibitem{Bolt88} J. Bolton, G. R. Jensen, M. Rigoli and L. M. Woodward, {\em On conformal minimal immersions of $\mathbb{S^2}$ into $\mathbb{C}P^n$}, Math. Ann., {\bf 279} (1988), 599--620.

\bibitem{Cala67} E. Calabi, {\em Minimal immersions of surfaces in euclidean spheres}, J. Differential Geom., {\bf 1} (1967), 111--125.

\bibitem{Chan93a} S. P. Chang, {\em On minimal hypersurfaces with constant scalar curvatures in $S^4$}, J. Differential Geom., {\bf 37} (1993), 523--534.

\bibitem{Chan93b} S. P. Chang, {\em A closed hypersurface with constant scalar and mean curvatures in $\mathbb{S}^4$ is isoparametric}, Comm. Anal. Geom., {\bf 1} (1993), 71--100.

\bibitem{Chen73a} B. Y. Chen, {\em On the surface with parallel mean curvature vector}, Indiana Univ. Math. J., {\bf 22} (1973), 655--666.

\bibitem{Chen73b} B. Y. Chen, {\em Geometry of submanifolds}, Mercer Dekker, New York, 1973.

\bibitem{Chen10} B. Y. Chen, {\em Submanifolds with parallel mean curvature vector in Riemannian and indefinite space forms}, Arab J. Math. Sci., {\bf 16} (2010), 1--46.

\bibitem{Chen90} Q. M. Cheng and H. Nakagawa, {\em Totally umbilic hypersurfaces}, Hiroshima Math. J., {\bf 20} (1990), 1--10.

\bibitem{Cher70} S. S. Chern, {\em On the minimal immersions of the two-sphere in a space of constant curvature}, in: Problems in Analysis, Princeton University Press, Princeton, 1970, 27--40.

\bibitem{Cher83} S. S. Chern, {\em On surfaces of constant mean curvature in a three-dimensional space of constant curvature}, in: Geometric Dynamics, Springer Lecture Notes (vol. 1007), 1983, 104--108.

\bibitem{CDK70} S. S. Chern, M. do Carmo, and S. Kobayashi, {\em Minimal submanifolds of a sphere with second fundamental form of constant length}, in: Functional Analysis and Related Fields, Springer, Berlin, 1970, 59--75.

\bibitem{Ding11} Q. Ding and Y. L. Xin, {\em On Chern's problem for rigidity of minimal hypersurfaces in the spheres}, Adv. Math., {\bf 227} (2011), 131--145.

\bibitem{Ding25} W. R. Ding, J. Q. Ge and F. G. Li, {\em Pinching rigidity of minimal surfaces in spheres}, Sci. China Math., {\bf 68} (2025), 2189--2206.

\bibitem{Ding26} W. R. Ding, J. Q. Ge and F. G. Li, {\em On Simon's third gap conjecture for minimal surfaces in spheres}, arXiv:2603.03070.

\bibitem{doCa71} M. do Carmo and N. Wallach, {\em Minimal immersions of spheres into spheres}, Ann. of Math., (2) {\bf 93} (1971), 43--62.

\bibitem{Guad83} I. Guadalupe and L. Rodriguez, {\em Normal curvature of surfaces in space forms}, Pacific J. Math., {\bf 106} (1983), 95--103.

\bibitem{Kozl84} M. Kozlowski and U. Simon, {\em Minimal immersion of $2$-manifolds into spheres}, Math. Z., {\bf 186} (1984), 377--382.

\bibitem{LeiX21} L. Lei, H. W. Xu, and Z. Y. Xu, {\em On the generalized Chern conjecture for hypersurfaces with constant mean curvature in a sphere}, Sci. China Math., {\bf 64} (2021), 1493--1504.

\bibitem{LiSi03} H. Li and U. Simon, {\em Quantization of curvature for compact surfaces in $S^n$}, Math. Z., {\bf 245} (2003), 201--216.

\bibitem{Peng83a} C. K. Peng and C. L. Terng, {\em Minimal hypersurfaces of spheres with constant scalar curvature}, in: Seminar on Minimal Submanifolds, Princeton University Press, Princeton, 1983, 179--198.

\bibitem{Peng83b} C. K. Peng and C. L. Terng, {\em The scalar curvature of minimal hypersurfaces in spheres}, Math. Ann., {\bf 266} (1983), 105--113.

\bibitem{Simo68} J. Simons, {\em Minimal varieties in Riemannian manifolds}, Ann. of Math. (2), {\bf 88} (1968), 62--105.

\bibitem{Tang23} Z. Z. Tang and W. J. Yan, {\em On the Chern conjecture for isoparametric hypersurfaces}, Sci. China Math., {\bf 66} (2023), 143--162.

\bibitem{XuXu13} H. W. Xu and Z. Y. Xu, {\em The second pinching theorem for hypersurfaces with constant mean curvature in a sphere}, Math. Ann., {\bf 356} (2013), 869--883.

\bibitem{XuXu17} H. W. Xu and Z. Y. Xu, {\em On Chern's conjecture for minimal hypersurfaces and rigidity of self-shrinkers}, J. Funct. Anal., {\bf 273} (2017), 3406--3425.

\bibitem{XuXu24} H. W. Xu and Z. Y. Xu, {\em The Chern conjecture for minimal hypersurfaces in a sphere and its related problems (in Chinese)}, Sci. Sin. Math., {\bf 54} (2024), 1723--1734.

\bibitem{XuXupre} H. W. Xu and Z. Y. Xu, {\em A pinching theorem for sectional curvature of PMC surfaces in a sphere}, preprint.

\bibitem{Yang98} H. C. Yang and Q. M. Cheng, {\em Chern's conjecture on minimal hypersurfaces}, Math. Z., {\bf 227} (1998), 377--390.

\bibitem{Yau74} S. T. Yau, {\em Submanifolds with constant mean curvature I}, Amer. J. Math., {\bf 96} (1974), 346--366.

\end{thebibliography}
\end{document}